\documentclass[10pt]{article}
\usepackage{amsfonts}
\usepackage{amssymb,amsmath,amsthm, amsfonts}
\usepackage[colorlinks,linkcolor=blue,anchorcolor=blue,citecolor=blue]{hyperref}
\newcommand{\esssup}{\mathop{\mathrm{esssup}}}
\newcommand{\essinf}{\mathop{\mathrm{essinf}}}
\usepackage{enumerate}

\newtheorem{theorem}{Theorem}[section]
\newtheorem{lemma}[theorem]{Lemma}
\newtheorem{definition}[theorem]{Definition}
\newtheorem{proposition}[theorem]{Proposition}
\newtheorem{corollary}[theorem]{Corollary}
\newtheorem{remark}[theorem]{Remark}

\numberwithin{equation}{section}

\allowdisplaybreaks
\begin{document}

\title{Nash equilibrium payoffs for stochastic  differential games with reflection}

\author{Qian Lin \footnote{ {\it Email address}:
linqian1824@163.com}
\\
{\small  Center for Mathematical Economics, Bielefeld University,
Postfach 100131, 33501 Bielefeld, Germany }}

\maketitle

\noindent{\bf Abstract}\ In this paper, we investigate Nash
equilibrium payoffs for nonzero-sum stochastic differential games
with reflection. We obtain an existence theorem and a
characterization theorem of Nash equilibrium payoffs for nonzero-sum
stochastic differential games with nonlinear cost functionals
defined by doubly controlled reflected backward stochastic
differential equations.

\vskip3mm


\section{Introduction}
In this paper, we study Nash equilibrium payoffs for nonzero-sum
stochastic differential games whose cost functionals are defined by
reflected  backward stochastic differential equations (RBSDEs, for
short). Fleming and Souganidis \cite{FS1989}
 were the first in a rigorous way to study zero-sum stochastic differential
 games. Since the pioneering work of Fleming and Souganidis
 \cite{FS1989},
 stochastic differential games have been investigated by many authors.
Recently, Buckdahn and Li \cite{BL2008} generalized  the
 results of Fleming and Souganidis \cite{FS1989}   by using a Girsanov transformation
argument  and a backward stochastic differential equation  (BSDE,
for short) approach.  The reader interested in this topic can be
referred to   Buckdahn, Cardaliaguet and Quincampoix \cite{BCQ2011},
Buckdahn and Li \cite{BL2008}, Fleming and Souganidis \cite{FS1989}
and the references therein.

El Karoui,  Kapoudjian,  Pardoux, Peng and Quenez \cite{EKPPQ1997}
introduced RBSDEs. By virtue of RBSDEs they gave a probabilistic
representation for the  solution of an obstacle problem for a
nonlinear parabolic partial differential equation. This kind of
RBSDEs also has many applications in finance, stochastic
differential games and stochastic optimal control problem. In
\cite{EPQ1997}, El Karoui, Pardoux and Quenez showed that the price
of an American option corresponds to the solution of a RBSDE.
Buckdahn and Li \cite{BL2007} considered zero-sum stochastic
differential games with reflection.
 Wu and Yu \cite{WY2008}
studied  one kind of stochastic recursive optimal control problem
with the obstacle constraint for the cost functional defined by a
RBSDE.

Buckdahn, Cardaliaguet and Rainer \cite{BCR2004} studied Nash
equilibrium payoffs for stochastic differential games. Recently, Lin
\cite{L2011,L2011a} studied Nash equilibrium payoffs for stochastic
differential games whose cost functionals are defined by doubly
controlled BSDEs.   Lin \cite{L2011,L2011a} generalizes the earlier
result by Buckdahn, Cardaliaguet and Rainer \cite{BCR2004}. In
\cite{L2011,L2011a},  the admissible control processes can depend on
events  occurring before the beginning of the stochastic
differential game, thus, the cost functionals are not necessarily
deterministic. Moreover, the  cost functionals are defined with the
help of BSDEs, and thus they are nonlinear. The objective of  this
paper is to  generalize the above results, i.e., investigate Nash
equilibrium payoffs for nonzero-sum stochastic differential games
with reflection. However, different from the earlier results by
Buckdahn, Cardaliaguet and Rainer \cite{BCR2004} and Lin
\cite{L2011,L2011a}, we shall study Nash equilibrium payoffs for
stochastic differential games with the running cost functionals
defined with the help of RBSDEs. For this, we first study the
properties of the value functions of stochastic differential games
with reflection. In comparison with Buckdahn and Li \cite{BL2007},
we shall study nonzero-sum stochastic differential games of the type
of {\it strategy against strategy}, while Buckdahn and Li
\cite{BL2007} considered the games of the type {\it strategy against
control}. Combining the arguments in Buckdahn, Cardaliaguet and
Quincampoix \cite{BCQ2011} and Buckdahn and Li \cite{BL2007}, we can
get the results in Section 4. Then we investigate Nash equilibrium
payoffs for stochastic differential games with reflection.  Our
results generalizes the results in Lin \cite{L2011} to the obstacle
constraint case. In Lin \cite{L2011}, the cost functionals of both
players do  not  have any obstacle constraint, so our results in
Section 5 are more general. The proof of our results is mainly based
on the techniques of mathematical analysis and the properties of
BSDEs with reflection. The presence of the obstacle constraint adds
us the difficulty of estimates and a supplementary complexity.

The paper is organized as follows. In Section 2, we
  introduce some notations and present some preliminary results
   concerning reflected backward stochastic differential equations,
    which we will need in what follows.
In Section 3, we introduce nonzero-sum stochastic differential games
with reflection and obtain the associated dynamic programming
principle. In Section 4 we give a probabilistic interpretation of
systems of Isaacs equations with obstacle.
  In Section 5, we obtain  the main results of this paper, i.e., an existence theorem
  and a characterization theorem of Nash equilibrium payoffs for nonzero-sum
stochastic differential games with reflection. In Section 6, we give
the proof of Theorem \ref{t1}.

\section{Preliminaries}\label{NS1}

The objective of this section is to recall some results about
RBSDEs, which are useful in what follows.   Let $B=\{B_t,\ 0\leq
t\leq T\}$ be a $d-$dimensional standard Brownian motion defined on
a probability space $(\Omega, \mathcal {F}, \mathbb{P})$. The
filtration $\mathbb{F}=\{\mathcal {F}_t,\ 0\leq t\leq T\}$ is
generated by  $B$ and augmented by all $\mathbb{P}-$null sets, i.e.,
\begin{eqnarray*}
\mathcal {F}_t=\sigma\Big\{B_r, 0\leq r \leq t \Big\}\vee\mathcal
{N}_{\mathbb{P}},
\end{eqnarray*}
where $\mathcal {N}_{\mathbb{P}}$ is the set of all
$\mathbb{P}-$null sets. Let us introduce some spaces:
\begin{eqnarray*}
&&\ L^2 (\Omega, \mathcal {F}_{T}, \mathbb{P}; \mathbb{R}^{n}) =
\bigg\{ \xi\ |\  \xi: \Omega\rightarrow\mathbb{R}^{n}\  \mbox {is
an} \ \mathcal {F}_{T} \mbox {-measurable random variable}  \mbox{
such
that}\ \mathbb{E}[|\xi|^2]<+\infty \bigg\},\\
&&\ S^2 (0,T; \mathbb{R})=\bigg\{ \varphi\ |\
\varphi:\Omega\times[0, T]\rightarrow\mathbb{R} \ \mbox {is a
predictable  process such that}  \ \mathbb{E}[\sup_{0\leq t\leq
T}|\varphi_t|^2]<+\infty \bigg\},\\
 &&\ \mathcal {H}^2 (0,T;
\mathbb{R}^{d}) =\bigg\{ \varphi\ |\ \varphi:\Omega\times[0,
T]\rightarrow\mathbb{R}^{d}\ \mbox { is a predictable process such
that } \mathbb{E}\int_0^T|\varphi_t|^2dt<+\infty \bigg\}.
\end{eqnarray*}

We consider  the following  one barrier reflected BSDE with data
$(f,\xi,S)$:
\begin{eqnarray}\label{e15}
\left\{
\begin{array}{rcl}
Y_t&=&\xi+\displaystyle\int_t^Tf(s,Y_s,Z_s)ds+K_T-K_t-\int_t^TZ_sdB_s,\\
Y_t&\geq& S_t, \quad \quad t\in[0,T],
\\ K_0&=&0, \hskip3mm \displaystyle\int_0^T(Y_r-S_r)dK_r=0,
\end{array}
\right.
\end{eqnarray}
where $\{K_t\}$ is an adapted, continuous and increasing process,
$f:\Omega\times [0,T]\times
\mathbb{R}\times\mathbb{R}^d\rightarrow\mathbb{R}$ and we make the
following assumptions:

 $(H2.1)$   $f(\cdot,0,0)\in \mathcal {H}^2 (0,T; \mathbb{R})$,

  $(H2.2)$
There exists  some constant $L>0$ such that for all
$y,y'\in\mathbb{R}$ and  $z,z'\in\mathbb{R}^d$,
\begin{eqnarray*}
|f(t,y,z)-f(t,y',z')|\leq L(|y-y'|+|z-z'|),
\end{eqnarray*}

 $(H2.3)$ $\{S_t\}_{t\in[0,T]}$ is a continuous process such that $\{S_{t}\}_{0\leq t\leq
T}\in S^2 (0,T; \mathbb{R})$.\vskip2mm

The following the existence and uniqueness theorem for solutions of
equation (\ref{e15}) was established in \cite{EKPPQ1997}.
\begin{lemma}
Under the assumptions $(H2.1)$-$(H2.3)$, if $ \xi \in L^2 (\Omega,
\mathcal {F}_{T}, \mathbb{P}; \mathbb{R})$ and  $S_{T}\leq \xi$,
then  equation (\ref{e15}) has a unique solution $(Y,Z,K)$.
\end{lemma}

We refer to \cite{EKPPQ1997} and \cite{WY2008} for the following two
estimates.
\begin{lemma}\label{l1}
Let the assumptions $(H2.1)$-$(H2.3)$ hold and let $(Y,Z,K)$ be the
solution of the
 reflected BSDE (\ref{e15}) with data $(\xi,g,S)$. Then there exists a positive constant $C$
such that
\[
\mathbb{E}[\sup_{t\leq s\leq T }
Y_s^2+\int_t^T|Z_s|^2+|K_T-K_t|^2\Big| \mathcal {F}_t]\leq
C\mathbb{E}[\xi^2+\Big(\int_t^T g(s,0,0)ds\Big)^2+\sup_{t\leq s\leq
T} S_s^2 \Big| \mathcal {F}_t].
\]
\end{lemma}

\begin{lemma}\label{l18}
We suppose that $(\xi,g,S)$ and
$(\xi^{\prime},g^{\prime},S^{\prime})$  satisfy the assumptions
$(H2.1)$-$(H2.3)$.  Let $(Y,Z,K)$ and
$(Y^{\prime},Z^{\prime},K^{\prime})$ be the solutions of the
reflected BSDEs (\ref{e15}) with data  $(\xi,g,S)$ and
$(\xi^{\prime},g^{\prime},S^{\prime})$, respectively. We let
\[
\Delta\xi=\xi-\xi^{\prime},\qquad \Delta g=g-g^{\prime},\qquad
\Delta S=S-S^{\prime},
\]
\[
\Delta Y=Y-Y^{\prime},\qquad \Delta Z=Z-Z^{\prime},\qquad \Delta
K=K-K^{\prime}.
\]
Then there exists a constant $C$ such that
\begin{eqnarray*}
&&\mathbb{E}\left[\sup_{t\leq s\leq T}|\Delta Y_s|^2
+\int_t^T|\Delta Z_s|^2ds +|\Delta K_T-\Delta K_t|^2\Big| \mathcal {F}_t\right]\\
&\leq& C\mathbb{E}\left[|\Delta\xi|^2+\left(\int_t^T |\Delta
g(s,Y_s,Z_s)|ds\right)^2\Big| \mathcal {F}_t
\right]+C\left(\mathbb{E}\left[\sup_{t\leq s\leq T}|\Delta
S_s|^2\Big| \mathcal {F}_t\right]\right)^{1/2}\Psi_{t,T}^{1/2},
\end{eqnarray*}
where
\begin{eqnarray*}
\Psi_{t,T} &=& \mathbb{E}\left[|\xi|^2+\left(\int_t^T|g(s,0,0)|ds
\right)^2 +\sup_{t\leq s\leq T} |S_s|^2\right.\\
&&\hskip 1.0cm \left.+|\xi^{\prime}|^2+\left(\int_t^T
|g^{\prime}(s,0,0)|ds \right)^2 +\sup_{t\leq s\leq T}
|S^{\prime}_s|^2\Big| \mathcal {F}_t\right].
\end{eqnarray*}
\end{lemma}

We also need the following lemma. For its proof, the interested
reader can refer to \cite{EKPPQ1997} and \cite{HLM1997} for more
details.
\begin{lemma}\label{l8}
Let us  denote by
 $(Y^{1},Z^{1},K^{1})$ and $(Y^{2},Z^{2},K^{2})$  the solutions of BSDEs
with data $(f^{1},\xi^{1},S^{1})$ and $(f^{2},\xi^{2},S^{2})$,
 respectively. If  $\xi^{1}, \xi^{2}\in L^2 (\Omega, \mathcal {F}_{T},
\mathbb{P};\mathbb{R})$, $S^{1}$ and $S^{2}$ satisfy  $(H2.3)$,  and
$f^{1}$ and $f^{2}$ satisfy the assumptions $(H2.1)$ and $(H2.2)$,
and the following holds

  (i) $\xi^{1}\leq \xi^{2}$, $\mathbb{P}-a.s.,$

  (ii)  $f^{1}(t,y_{t}^{2}, z_{t}^{2}) \leq f^{2}(t, y_{t}^{2}, z_{t}^{2})$,
  $dtd\mathbb{P}-a.e.,$

  (iii) $S^{1}\leq S^{2}$, $\mathbb{P}-a.s.$\\
Then, we have $Y_{t}^{1} \leq Y_{t}^{2}$, $ a.s.$, for all $t \in
[0,T]$. Moreover, if

 (iv)  $f^{1}(t,y, z) \leq
f^{2}(t, y, z), (t,y, z)\in [0,T]\times \mathbb{R} \times
\mathbb{R}^{d},$  $dtd\mathbb{P}-a.e.,$

  (v) $S^{1}= S^{2}$, $\mathbb{P}-a.s.$\\
Then, $K^{1}_{t}\geq K^{2}_{t}$, $\mathbb{P}-a.s.,$ for all
$t\in[0,T],$ and $\{K^{1}_{t}- K^{2}_{t}\}_{t\in[0,T]}$ is a
increasing process.

\end{lemma}

\section{Nonzero-sum stochastic differential games with reflection and associated dynamic programming principle}

In what follows, we assume that  $U$ and  $V$ are two compact metric
spaces. The space $U$ is considered as the control state space of
the first player, and $V$ as that of the second one. We denote the
associated sets of admissible controls  by $\mathcal {U}$ and
$\mathcal {V}$, respectively. The set $\mathcal {U}$ is formed by
all $U$-valued $\mathbb{F}$-progressively measurable processes, and
$\mathcal {V}$ is the set of all $V$-valued
$\mathbb{F}$-progressively measurable processes.

For  given admissible controls $u(\cdot)\in\mathcal {U}$ and
$v(\cdot)\in\mathcal {V}$, let us consider the following control
system: for $t\in[0,T],$
\begin{equation}\label{eq1}
\left\{
\begin{array}{rcl}
dX^{t,x;u,v}_s &=& b(s,X^{t,x;u,v}_s,
u_s,v_s)ds+\sigma(s,X^{t,x;u,v}_s,u_s,v_s)dB_s,\qquad
s\in [t,T],\\
X^{t,x;u,v}_t &=& x\in\mathbb{R}^n,
\end{array}
\right.
\end{equation}
where
\[
b:[0,T]\times\mathbb{R}^n\times U\times
V\rightarrow\mathbb{R}^n,\qquad \sigma:[0,T]\times\mathbb{R}^n\times
U\times V\rightarrow\mathbb{R}^{n\times d}.
\]
We make the following assumptions:
\begin{list}{}{\setlength{\itemindent}{0cm}}
\item[$(H3.1)$] For all $x\in\mathbb{R}^{n}$, $b(\cdot,x,\cdot,\cdot)$ and $\sigma(\cdot,x,\cdot,\cdot)$
are continuous in $(t,u,v)$.
\item[$(H3.2)$] There exists a positive constant $L$ such that, for all $t\in [0,T], x,x^{\prime}\in\mathbb{R}^n$,
$u\in U, v\in V$,
\[
|b(t,x,u,v)-b(t,x^{\prime},u,v)|+|\sigma(t,x,u,v)-\sigma(t,x^{\prime},u,v)|\leq
L|x-x^{\prime}|.
\]
\end{list}
Under the above assumptions, for any $u(\cdot)\in \mathcal {U}$ and
$v(\cdot)\in \mathcal {V}$, the control system (\ref{eq1}) has a
unique strong solution $\{X^{t,x;u,v}_s,\ 0\le t\le s\le T\}$, and
we also have the following standard estimates for solutions.
\begin{lemma}\label{l4}
For all $p\geq 2$, there exists a  positive constant $C_{p}$ such
that, for all $t\in [0,T]$, $x, x^{\prime}\in \mathbb{R}^n$,
$u(\cdot)\in\mathcal {U}$ and $v(\cdot)\in \mathcal {V}$,
\begin{eqnarray*}
&&\mathbb{E}\left[\sup_{t\leq s\leq T}|X^{t,x;u,v}_s|^p\Big|\mathcal
{F}_t\right]\leq
C_{p}(1+|x|^p),\quad \mathbb{P}-a.s.,\\
&& \mathbb{E}\left[\sup_{t\leq s\leq
T}|X^{t,x;u,v}_s-X^{t,x';u,v}_s|^p\Big|\mathcal {F}_t\right] \leq
C_{p}|x-x^{\prime}|^p ,\quad \mathbb{P}-a.s.,
\end{eqnarray*}
where the constant $C_{p}$ only depends  on $p$, the Lipschitz
constant and the linear growth of $b$ and $\sigma$.
\end{lemma}

For  given admissible controls $u(\cdot)\in\mathcal {U}$ and
$v(\cdot)\in\mathcal {V}$, let us  consider the following doubly
controlled RBSDEs:  for $j=1,2,$
\begin{equation}\label{BSDE}
\left\{
\begin{array}{lll}
&^{j}Y^{t,x;u,v}_s = \Phi_{j}(X^{t,x;u,v}_T)
+\displaystyle \int_s^T f_{j}(r,X^{t,x;u,v}_r,\ ^{j}Y^{t,x;u,v}_r,\ ^{j}Z^{t,x;u,v}_r,u_r, v_r)dr\\
&\hskip2cm +\ ^{j}K^{t,x;u,v}_T -\ ^{j}K^{t,x;u,v}_s-\displaystyle
\int_s^T\
^{j}Z^{t,x;u,v}_rdB_r,\qquad s\in[t,T],\\
& ^{j}Y^{t,x;u,v}_s \geq h_{j}(s,X^{t,x;u,v}_s), \ s\in[t,T],\\
& ^{j}K^{t,x;u,v}_t=0,\quad \displaystyle \int_t^T
(^{j}Y^{t,x;u,v}_r -h_{j}(r,X^{t,x;u,v}_r))d\ ^{j}K^{t,x;u,v}_r=0,
\end{array}
\right.
\end{equation}
where $X^{t,x;u,v}$ is introduced in equation (\ref{eq1}) and
\begin{eqnarray*}
 \Phi_{j}=\Phi_{j}(x):\mathbb{R}^n\rightarrow\mathbb{R},
 \quad  \quad
h_{j}=h_{j}(t,x):[0,T]\times\mathbb{R}^n\rightarrow \mathbb{R},\\
f_{j}=f_{j}(t,x,y,z,u,v):[0,T]\times\mathbb{R}^n\times\mathbb{R}\times\mathbb{R}^d\times
U\times V\rightarrow \mathbb{R}.
\end{eqnarray*}
We make the following assumptions:
\begin{list}{}{\setlength{\itemindent}{0cm}}
\item[$(H3.3)$] There exists a positive constant $L$ such that, for all $t\in[0,T], x,x^{\prime}\in\mathbb{R}^n$,
$y,y^{\prime}\in\mathbb{R}$, $z,z^{\prime}\in\mathbb{R}^d$, $u\in U$
and $ v\in V$, $\Phi_{j}(x)\geq h_{j}(T,x)$ and
\begin{eqnarray*}
&&|f_{j}(t,x,y,z,u,v)-f_{j}(t,x^{\prime},y^{\prime},z^{\prime},u,v)|
+|\Phi_{j}(x)-\Phi_{j}(x^{\prime})|+|h_{j}(t,x)-h_{j}(t,x^{\prime})|\\
&&\leq L(|x-x^{\prime}|+|y-y^{\prime}|+|z-z^{\prime}|).
\end{eqnarray*}
\item[($H3.4$)] For all $(x,y,z)\in\mathbb{R}^{n}\times\mathbb{R}\times\mathbb{R}^{d}$,
 $f_{j}(\cdot,x,y,z,\cdot,\cdot)$ is  continuous in
 $(t,u,v)$, and there exist   positive constants $C$ and $\alpha\geq \frac{1}{2}$ such that, for all $t,s\in[0,T],
 x\in\mathbb{R}^n$,
 \begin{eqnarray*}
|h_{j}(t,x)-h_{j}(s,x)|\leq C|t-s|^{\alpha}.
\end{eqnarray*}
\end{list}

Under the assumption (H3.3),  from \cite{EKPPQ1997} we know that
equation (\ref{BSDE}) admits a unique solution.
 For given  control processes $u(\cdot)\in U$ and $v(\cdot)\in V$, let us introduce now the associated
 cost functional for
 player $j$, $j=1,2,$
\begin{equation*}\label{}
J_{j}(t,x;u,v):= \left.\ ^{j}Y^{t,x;u,v}_s\right|_{s=t},\qquad
(t,x)\in [0,T]\times\mathbb{R}^n.
\end{equation*}

\vskip1mm From Buckdahn and Li \cite{BL2007} we have  the following
estimates for solutions.
\begin{proposition}\label{p3}
Under the assumption (H3.1)-(H3.3), there exists a  positive
constant $C$ such that, for all $t\in [0,T]$, $u(\cdot)\in\mathcal
{U}$ and $v(\cdot)\in \mathcal {V}$, $x,x^{\prime}\in \mathbb{R}^n$,
\begin{eqnarray*}
&&|^{j}Y^{t,x;u,v}_t|\leq C(1+|x|), \quad \mathbb{P}-a.s.,\\
&&|^{j}Y^{t,x;u,v}_t-\ ^{j}Y^{t,x^{\prime};u,v}_{t}| \leq
C|x-x^{\prime}|,\quad \mathbb{P}-a.s.
\end{eqnarray*}
\end{proposition}

We now recall the definition of admissible controls and NAD
strategies, which was introduced in \cite{L2011}.
\begin{definition}
The space $\mathcal {U}_{t,T}$ (resp., $\mathcal {V}_{t,T}$) of
admissible controls for $1^{th}$ player (resp., $2^{nd}$) on the
interval $[t, T]$  is defined as the space of all processes
$\{u_{r}\}_{r\in[t,T]}$ (resp., $\{v_{r}\}_{r\in[t,T]}$), which are
$\mathbb{F}$-progressively measurable and take values in $U$ (resp.,
$V$).
\end{definition}

\begin{definition}\label{d1}
 A nonanticipative strategy with delay (NAD strategy)
for $1^{th}$ player is a measurable mapping $\alpha:\mathcal
{V}_{t,T}\rightarrow \mathcal {U}_{t,T}$, which satisfies the
following properties:

1)  $\alpha$ is a nonanticipative strategy, i.e., for every
$\mathbb{F}$-stopping time $\tau:\Omega\rightarrow [t,T],$ and for
 $v_{1},v_{2} \in\mathcal {V}_{t,T}$ with $v_{1}=v_{2}$   on $[[ t,\tau]]$,
it holds $\alpha(v_{1})=\alpha(v_{2})$ on $[[t,\tau]]$. (Recall that
$[[ t,\tau]]=\{(s,\omega)\in[t,T]\times \Omega, t\leq
s\leq\tau(\omega) \}$).

2) $\alpha$ is a   strategy with delay, i.e., for all $v\in\mathcal
{V}_{t,T}$,
 there exists an increasing sequence of stopping times
$\{S_{n}(v)\}_{n\geq 1}$ with

i) $t=S_{0}(v)\leq S_{1}(v) \leq\cdots \leq S_{n}(v)\leq \cdots \leq
T$,

ii) $\bigcup_{n\geq
1}\{S_{n}(v)=T\}=\Omega$, $\mathbb{P}$-a.s.,\\
such that, for all $n\geq 1 $ and $v,v' \in \mathcal {V}_{t,T},
\Gamma\in \mathcal {F}_{t}$, it holds:  if $v=v'$
 on $[[ t,S_{n-1}(v)]]\bigcap([t,T]\times\Gamma)$, then

iii) $S_{l}(v)=S_{l}(v')$, on $\Gamma,$ $1\leq l \leq n$,

iv) $\alpha(v)=\alpha(v')$, on $[[
t,S_{n}(v)]]\bigcap([t,T]\times\Gamma)$.\vskip2mm

We denote the set of all NAD strategies for $1^{th}$ player for
games over the time interval $[t,T]$  by $\mathcal {A}_{t,T}$. The
set of all NAD strategies $\beta:\mathcal {U}_{t,T}\rightarrow
\mathcal {V}_{t,T}$ for $2^{nd}$ player for games over the time
interval $[t,T]$ is defined in a symmetrical way and we denote it by
$\mathcal {B}_{t,T}$.
\end{definition}

NAD strategy allows us to put stochastic differential games under
normal form. The following lemma was established in \cite{L2011}.
\begin{lemma}\label{l2}
If  $(\alpha,\beta)\in \mathcal {A}_{t,T}\times \mathcal {B}_{t,T}$,
then there exists a unique couple of admissible control $(u,v)\in
\mathcal {U}_{t,T} \times \mathcal {V}_{t,T}$ such that
\begin{eqnarray*}\label{}
\alpha(v)=u, \quad \beta(u)=v.
\end{eqnarray*}
\end{lemma}
If $(\alpha,\beta)\in\mathcal {A}_{t,T}\times \mathcal {B}_{t,T}$,
then from Lemma \ref{l2} we  have a unique couple $(u,v)\in\mathcal
{U}_{t,T}\times\mathcal {V}_{t,T}$ such that
$(\alpha(v),\beta(u))=(u,v)$. Then let us put
$J_{j}(t,x;\alpha,\beta)=J_{j}(t,x;u,v)$. Therefore, let us define:
for all $(t,x)\in [0,T]\times\mathbb{R}^n$,
\begin{eqnarray*}\label{}
W_{j}(t,x):=\esssup_{\alpha\in\mathcal {A}_{t,T}}
\essinf_{\beta\in\mathcal {B}_{t,T}} J_{j}(t,x;\alpha,\beta),
\end{eqnarray*}
and
\begin{eqnarray*}\label{}
U_{j}(t,x):=\essinf_{\beta\in\mathcal {B}_{t,T}}
\esssup_{\alpha\in\mathcal {A}_{t,T}} J_{j}(t,x;\alpha,\beta).
\end{eqnarray*}
 Under the assumptions (H3.1)--(H3.3) we see that  $W_{j}(t,x)$ and $
U_{j}(t,x)$ are random variables. But using the arguments in
\cite{BCQ2011} and \cite{L2011a}, we have the following proposition.
\begin{proposition}
Under the assumptions (H3.1)--(H3.3), for all
$(t,x)\in[0,T]\times\mathbb{R}^n$, the value functions $W_{j}(t,x)$
and $ U_{j}(t,x)$ are deterministic functions.
\end{proposition}

Let us  now recall the definition  of stochastic backward
semigroups, which was first introduced by Peng \cite{P1997}  to
study stochastic optimal control problem. For a given initial
condition $(t,x)\in [0,T]\times\mathbb{R}^n$,  $0\leq\delta\leq
T-t$, for admissible control processes $u(\cdot)\in\mathcal
{U}_{t,t+\delta}$ and $v(\cdot)\in\mathcal {V}_{t,t+\delta}$, and a
real-valued random variable $\eta\in L^2(\Omega,\mathcal
{F}_{t+\delta},\mathbb{P};\mathbb{R})$ such that $ \eta\geq
h_{j}(t+\delta,X^{t,x;u,v}_{t+\delta})$, we define
\[
^{j}G^{t,x;u,v}_{t,t+\delta}[\eta]:=\ ^{j}\overline{Y}_t^{t,x;u,v},
\]
where $(^{j}\overline{Y}^{t,x;u,v},\ ^{j}\overline{Z}^{t,x;u,v},\
^{j}\overline{K}^{t,x;u,v})$ is the unique solution of the following
reflected BSDE over the time interval $[t,t+\delta]$:
\begin{equation*}
\left\{
\begin{array}{lll}
&^{j}\overline{Y}^{t,x;u,v}_s = \eta +\displaystyle
\int_s^{t+\delta}
f_{j}(r,X^{t,x;u,v}_r,\ ^{j}\overline{Y}^{t,x;u,v}_r,\ ^{j}\overline{Z}^{t,x;u,v}_r,u_r, v_r)dr\\
&\hskip2cm +\ ^{j}\overline{K}^{t,x;u,v}_{t+\delta} -\
^{j}\overline{K}^{t,x;u,v}_s-\displaystyle \int_s^{t+\delta}\
^{j}\overline{Z}^{t,x;u,v}_rdB_r,\quad s\in[t,t+\delta],\\
& ^{j}\overline{Y}^{t,x;u,v}_s \geq h_{j}(s,X^{t,x;u,v}_s), \qquad
s\in[t,t+\delta],
\\
& ^{j}\overline{K}^{t,x;u,v}_t=0,\quad
\displaystyle\int_t^{t+\delta} (^{j}\overline{Y}^{t,x;u,v}_r
-h_{j}(r,X^{t,x;u,v}_r))d^{j}\overline{K}^{t,x;u,v}_r=0,
\end{array}
\right.
\end{equation*}
and $X^{t,x;u,v}$ is the unique solution of equation (\ref{eq1}).

For $(t,x)\in[0,T]\times\mathbb{R}^{n},$ $(u,v)\in \mathcal
{U}_{t,T}\times \mathcal {V}_{t,T}, 0\leq \delta \leq T-t, j=1,2,$
we have
\begin{eqnarray*}\label{}
J_{j}(t,x;u,v)&=&
\ ^{j}G^{t,x;u,v}_{t,T}[\Phi_{j}(X^{t,x;u,v}_T)]=\ ^{j}G^{t,x;u,v}_{t,t+\delta}[^{j}Y^{t,x;u,v}_{t+\delta}]\\
&=&^{j}G^{t,x;u,v}_{t,t+\delta}[J_{j}(t+\delta,X^{t,x;u,v}_{t+\delta},u,v)].
\end{eqnarray*}

\begin{remark}\label{}
We consider a special case of $f_{j}$. If  $f_{j}$ is independent of
$(y,z)$, then we have
\begin{eqnarray*}\label{}
J_{j}(t,x;u,v)= \ ^{j}G^{t,x;u,v}_{t,t+\delta}[\eta]
=\mathbb{E}\Big[\eta+\displaystyle\int_{t}^{t+\delta}f_{j}(r,X^{t,x;u,v}_r,u_r,
v_r)dr+ \ ^{j}\overline{K}^{t,x;u,v}_{t+\delta}\ \Big |\ \mathcal
{F}_{t}\Big].
\end{eqnarray*}

\end{remark}

\begin{proposition}\label{p1}
Under the assumptions (H3.1)--(H3.3)  we have the following dynamic
programming principle: for all $0<\delta\leq T-t,
x\in\mathbb{R}^{n}$,
\begin{equation*}
W_{j}(t,x)=\esssup_{\alpha\in\mathcal
{A}_{t,t+\delta}}\essinf_{\beta\in\mathcal {B}_{t,t+\delta}} \
^{j}G^{t,x;\alpha,\beta}_{t,t+\delta}[W_{j}(t+\delta,X^{t,x;\alpha,\beta}_{t+\delta})],
\end{equation*}
and
\begin{equation*}
U_{j}(t,x)=\essinf_{\beta\in\mathcal
{B}_{t,t+\delta}}\esssup_{\alpha\in\mathcal {A}_{t,t+\delta}} \
^{j}G^{t,x;\alpha,\beta}_{t,t+\delta}[U_{j}(t+\delta,X^{t,x;\alpha,\beta}_{t+\delta})].
\end{equation*}
\end{proposition}
The proof of the above proposition is similar to \cite{BCQ2011} and
\cite{L2011a}, we omit the proof here.

\begin{proposition} \label{pro}
Under the assumptions (H3.1)--(H3.4),  there exists a positive
constant $C$ such that, for all $t,t'\in [0,T]$ and $x,
x^{\prime}\in\mathbb{R}^n$, we have
\begin{enumerate}[(i)]
\item $W_{j}(t,x)$ is $\frac{1}{2}$-H\"{o}lder continuous in $t$:
$$|W_{j}(t,x)-W_{j}(t',x)|\leq C(1+|x|)|t-t^{\prime}|^{\frac{1}{2}};$$
\item $|W_{j}(t,x)-W_{j}(t,x^{\prime})|\leq C|x-x^{\prime}|.$
\end{enumerate}
The same properties hold true for the function $U_{j}$.
\end{proposition}
By means of the standard arguments and Proposition \ref{p1} we can
easily get the above proposition. The proof of the above proposition
is omitted here.

\section{Probabilistic interpretation of  systems of Isaacs equations with obstacle }

The objective of  this section is to  give a probabilistic
interpretation of systems of Isaacs equations with obstacle, and
show that $W_{j}$  and  $U_{j}$  introduced in Section 3, are the
viscosity solutions of the following  Isaacs equations with
obstacle, for $(t,x)\in [0,T)\times {\mathbb{R}}^n$,
\begin{eqnarray}\label{e25}
\left\{
\begin{array}{rcl}
\min\Big\{W_{j}(t,x)-h_{j}(t,x),\  -\dfrac{\partial }{\partial t}
W_{j}(t,x) - H_{j}^{-}(t, x, W_{j}(t,x),DW_{j}(t,x),
D^2W_{j}(t,x))\Big\}=0,\\
 W_{j}(T,x)=\Phi_{j}(x),
 \end{array}
\right.
\end{eqnarray}
and
\begin{eqnarray}\label{e26}
\left\{
\begin{array}{rcl}
\min\Big\{U_{j}(t,x)-h_{j}(t,x), \ -\dfrac{\partial }{\partial t}
U_{j}(t,x) - H_{j}^{+}(t, x, U_{j}(t,x),DU_{j}(t,x),
D^2U_{j}(t,x))\Big\}=0,\\
 U_{j}(T,x)=\Phi_{j}(x),
 \end{array}
\right.
\end{eqnarray}
respectively,  where
\begin{eqnarray*}
H_{j}(t, x, y, p, A, u,v)= \dfrac{1}{2}tr(\sigma\sigma^{T}(t, x, u,
v) A)+ p^{T}b(t, x, u, v)+ f_{j}(t, x, y,  p^{T}\sigma(t, x, u, v),
u, v),
\end{eqnarray*}
$ (t, x, y,  p,  u,v)\in [0, T]\times{\mathbb{R}}^n\times
\mathbb{R}\times \mathbb{R}^{n}\times U\times V$ and $ A\in
\mathbb{S}^{n}$ ($\mathbb{S}^{n}$ denotes all the $n\times n$
symmetric matrices),
\begin{eqnarray*}
H_{j}^-(t, x, y,  p, A)= \sup_{u \in U}\inf_{v \in V}H_{j}(t, x, y,
p, A, u,v),
\end{eqnarray*}
and
\begin{eqnarray*}
H_{j}^+(t, x, y,  p, A)= \inf_{v \in V}\sup_{u \in U}H_{j}(t, x, y,
p, A, u,v).
\end{eqnarray*}

We denote by $C^3_{l, b}([0,T]\times {\mathbb{R}}^n)$ the set of
real-valued functions which are continuously differentiable up to
the third order and whose derivatives of order from 1 to 3 are
bounded. Let us recall the definition of a viscosity solution  of
(\ref{e25}). The definition of  a viscosity solution of (\ref{e26})
can be defined in a similar way.
 \begin{definition}\label{d3}
For fixed $j=1,2,$  let  $w_{j}\in C([0,T]\times \mathbb{R}^n;\mathbb{R})$ be a function.  It is called \\
  {\rm(i)} a viscosity subsolution of  (\ref{e25}) if
\begin{eqnarray*}
w_{j}(T,x) \leq \Phi_{j} (x),\ \mbox{for all}  \ x \in
{\mathbb{R}}^n,
\end{eqnarray*}
and if for all functions $\varphi \in C^3_{l, b}([0,T]\times
  {\mathbb{R}}^n), $ and $(t,x) \in [0,T) \times {\mathbb{R}}^n$ such that $w_{j}-\varphi $\ attains
  a local maximum at $(t, x)$,
\begin{eqnarray*}
\min\Big\{w_{j}(t,x)-h_{j}(t,x),\  -\dfrac{\partial }{\partial t}
\varphi(t,x) - H_{j}^{-}(t, x, w_{j}(t,x),D\varphi(t,x),
D^2\varphi(t,x))\Big\}\leq 0,
\end{eqnarray*}

\noindent{\rm(ii)} a viscosity supersolution of  (\ref{e25}) if
\begin{eqnarray*}
w_{j}(T,x) \geq \Phi_{j} (x),\ \mbox{for all}\   x \in
{\mathbb{R}}^n,
\end{eqnarray*}
and if for all functions $\varphi \in C^3_{l, b}([0,T]\times
  {\mathbb{R}}^n), $ and $(t,x) \in [0,T) \times {\mathbb{R}}^n$ such that $w_{j}-\varphi $\ attains
  a local minimum at $(t, x)$,
\begin{eqnarray*}
\min\Big\{w_{j}(t,x)-h_{j}(t,x),\  -\dfrac{\partial }{\partial t}
\varphi(t,x) - H_{j}^{-}(t, x, w_{j}(t,x),D\varphi(t,x),
D^2\varphi(t,x))\Big\}\geq 0,
\end{eqnarray*}
 {\rm(iii)} a viscosity solution of (\ref{e25}) if it is both a viscosity subsolution
  and a supersolution of   (\ref{e25}).
\end{definition}

We adapt the methods in Buckdahn and Li \cite{BL2007} and Buckdahn,
Cardaliaguet and Quincampoix \cite{BCQ2011} to our framework. We can
obtain the following theorem.

\begin{theorem}\label{t1}
Under the assumptions (H3.1)--(H3.3), the function
  $W_{j}$ (resp., $U_{j}$) is a viscosity solution of the system
(\ref{e25}) (resp., (\ref{e26})).
\end{theorem}

Let us  now give a comparison  theorem for the viscosity solution of
(\ref{e25}) and (\ref{e26}).  We first introduce the following
space:
 \begin{eqnarray*}
 \Theta:&=&\Big\{ \varphi\in C([0, T]\times {\mathbb{R}}^n): \mbox{there exists a constant}\ \ A>0\ \mbox{such
 that}\\ && \qquad \lim_{|x|\rightarrow \infty}|\varphi(t, x)|\exp\{-A[\log((|x|^2+1)^{\frac{1}{2}})]^2\}=0,\
 \mbox{uniformly in}\ t\in [0, T]\Big\}.
\end{eqnarray*}

 \begin{theorem}\label{t3}
 Under the assumptions (H3.1)--(H3.3), if an upper semicontinuous
 function $u_{1}\in  \Theta$ is a viscosity subsolution   of  (\ref{e25})
 (resp., (\ref{e26})), and a lower semicontinuous function $u_{2}\in  \Theta$ is a
 viscosity supersolution   of  (\ref{e25}) (resp.,
 (\ref{e26})), then we have the following:
\begin{eqnarray*}
u_{1}(t,x)\leq u_{2}(t,x), \ \ \mbox{for all}\ (t,x)\in[0,T]\times
\mathbb{R}^{n}.
\end{eqnarray*}
\end{theorem}
By means of the  arguments in  Buckdahn  and Li \cite{BL2007}, we
can give the proof of this theorem and the proof is omitted here.

\begin{remark}
By Proposition \ref{pro} we see that   $W_{j}$ (resp., $U_{j}$) is a
viscosity solution of linear growth. Therefore, from the above
theorem we know that  $W_{j}$ (resp., $U_{j}$) is the unique
viscosity solution in $\Theta$ of the system (\ref{e25}) (resp.,
(\ref{e26})).
\end{remark}

 \noindent {\bf Isaacs condition:}\vskip2mm

 For all
$ (t, x, y, p, A, u,v)\in [0, T]\times{\mathbb{R}}^n\times\mathbb{R}
\times \mathbb{R}^{n}\times \mathbb{S}^{n}\times U\times V,$
$j=1,2,$ we have
\begin{eqnarray}\label{Isaacs}
\begin{aligned}
&\sup\limits_{u\in U}\inf\limits_{v\in V}
\Big\{\dfrac{1}{2}tr(\sigma\sigma^{T}(t,
x, u, v) A)+ p^{T} b(t, x, u, v)+ f_{j}(t, x, y,  p^{T}\sigma(t, x, u, v), u, v) \Big\}\\
=&\inf\limits_{v\in V}\sup\limits_{u\in U}
\Big\{\dfrac{1}{2}tr(\sigma\sigma^{T}(t, x, u, v) A)+ p^{T} b(t, x,
u, v)+ f_{j}(t, x, y,  p^{T}\sigma(t, x, u, v), u, v) \Big\}.
 \end{aligned}
\end{eqnarray}

\begin{corollary}\label{c1}
Let Isaacs condition (\ref{Isaacs}) hold. Then we have,  for all $
(t, x)\in [0, T]\times{\mathbb{R}}^n$,
\begin{eqnarray*}
(U_{1}(t,x),U_{2}(t,x))=(W_{1}(t,x),W_{2}(t,x)).
\end{eqnarray*}
\end{corollary}

In a symmetric way, for all $(t,x)\in [0,T]\times\mathbb{R}^n$, we
put
\begin{eqnarray*}\label{}
\overline{W}_{j}(t,x):=\esssup_{\beta\in\mathcal {B}_{t,T}}
\essinf_{\alpha\in\mathcal {A}_{t,T}} J_{j}(t,x;\alpha,\beta),
\end{eqnarray*}
and
\begin{eqnarray*}\label{}
\overline{U}_{j}(t,x):=\essinf_{\alpha\in\mathcal {A}_{t,T}}
\esssup_{\beta\in\mathcal {B}_{t,T}} J_{j}(t,x;\alpha,\beta).
\end{eqnarray*}
 Using the arguments in
\cite{BCQ2011} and \cite{L2011a}, we have the following
propositions.
\begin{proposition}
Under the assumptions (H3.1)--(H3.3), for all
$(t,x)\in[0,T]\times\mathbb{R}^n$, the value functions
$\overline{W}_{j}(t,x)$ and $ \overline{U}_{j}(t,x)$ are
deterministic functions.
\end{proposition}

\begin{proposition}\label{pp1}
Under the assumptions (H3.1)--(H3.3)  we have the following dynamic
programming principle: for all $0<\delta\leq T-t,
x\in\mathbb{R}^{n}$,
\begin{equation*}
\overline{W}_{j}(t,x)=\esssup_{\beta\in\mathcal
{B}_{t,t+\delta}}\essinf_{\alpha\in\mathcal {A}_{t,t+\delta}} \
^{j}G^{t,x;\alpha,\beta}_{t,t+\delta}[\overline{W}_{j}(t+\delta,X^{t,x;\alpha,\beta}_{t+\delta})],
\end{equation*}
and
\begin{equation*}
\overline{U}_{j}(t,x)=\essinf_{\alpha\in\mathcal
{A}_{t,t+\delta}}\esssup_{\beta\in\mathcal {B}_{t,t+\delta}} \
^{j}G^{t,x;\alpha,\beta}_{t,t+\delta}[\overline{U}_{j}(t+\delta,X^{t,x;\alpha,\beta}_{t+\delta})].
\end{equation*}
\end{proposition}
 \noindent {\bf Isaacs condition:}\vskip2mm

 For all
$ (t, x, y, p, A, u,v)\in [0, T]\times{\mathbb{R}}^n\times
\mathbb{R}\times \mathbb{R}^{n}\times \mathbb{S}^{n}\times U\times
V,$  $j=1,2,$ we have
\begin{eqnarray}\label{Isaacs1}
\begin{aligned}
&\inf\limits_{u\in U}\sup\limits_{v\in V}
\Big\{\dfrac{1}{2}tr(\sigma\sigma^{T}(t,
x, u, v) A)+ p^{T} b(t, x, u, v)+ f_{j}(t, x, y,  p^{T}\sigma(t, x, u, v), u, v) \Big\}\\
=&\sup\limits_{v\in V}\inf\limits_{u\in U}
\Big\{\dfrac{1}{2}tr(\sigma\sigma^{T}(t, x, u, v) A)+ p^{T} b(t, x,
u, v)+ f_{j}(t, x, y,  p^{T}\sigma(t, x, u, v), u, v) \Big\}.
 \end{aligned}
\end{eqnarray}

By virtue of arguments in this section, we have the following
proposition.
\begin{proposition}\label{cc1}
Let Isaacs condition (\ref{Isaacs1}) hold. Then we have,  for all $
(t, x)\in [0, T]\times{\mathbb{R}}^n$,
\begin{eqnarray*}
(\overline{U}_{1}(t,x),\overline{U}_{2}(t,x))=(\overline{W}_{1}(t,x),\overline{W}_{2}(t,x)).
\end{eqnarray*}
\end{proposition}

\section{Nash equilibrium payoffs}
The objective of this section is to obtain an existence of a Nash
equilibrium payoff. For this, we consider two zero-sum stochastic
differential games associated with $J_{1}$ and $J_{2}$, i.e., the
first player wants to maximize  $J_{1}$ and the second player wants
to minimize $J_{1}$, while the first player wants to minimize
$J_{2}$ and the second player wants to maximize $J_{2}$.

In what follows, we redefine the following notations which are
different from the above sections: for
$(t,x)\in[0,T]\times\mathbb{R}^{n}$,
\begin{eqnarray*}\label{}
W_{1}(t,x):= \esssup_{\alpha\in\mathcal {A}_{t,T}}
\essinf_{\beta\in\mathcal {B}_{t,T}} J_{1}(t,x;\alpha,\beta),\ \
W_{2}(t,x):=\esssup_{\beta\in\mathcal
{B}_{t,T}}\essinf_{\alpha\in\mathcal {A}_{t,T}}
J_{2}(t,x;\alpha,\beta)
\end{eqnarray*}

We suppose that the following holds:\vskip2mm

\noindent {\bf Isaacs condition A:}\vskip2mm
 For all
$ (t, x, y, p, A, u,v)\in [0,
T]\times{\mathbb{R}}^n\times\mathbb{R}\times \mathbb{R}^{n}\times
\mathbb{S}^{n}\times U\times V,$   we have
\begin{eqnarray*}\label{}
\begin{aligned}
&\sup\limits_{u\in U}\inf\limits_{v\in V}
\Big\{\dfrac{1}{2}tr(\sigma\sigma^{T}(t,
x, u, v) A)+ p^{T} b(t, x, u, v)+ f_{1}(t, x, y,  p^{T}\sigma(t, x, u, v), u, v) \Big\}\\
=&\inf\limits_{v\in V}\sup\limits_{u\in U}
\Big\{\dfrac{1}{2}tr(\sigma\sigma^{T}(t, x, u, v) A)+ p^{T} b(t, x,
u, v)+ f_{1}(t, x, y,  p^{T}\sigma(t, x, u, v), u, v) \Big\},
 \end{aligned}
\end{eqnarray*}
and
\begin{eqnarray*}\label{}
\begin{aligned}
&\inf\limits_{u\in U}\sup\limits_{v\in V}
\Big\{\dfrac{1}{2}tr(\sigma\sigma^{T}(t,
x, u, v) A)+ p^{T} b(t, x, u, v)+ f_{2}(t, x, y,  p^{T}\sigma(t, x, u, v), u, v) \Big\}\\
=&\sup\limits_{v\in V}\inf\limits_{u\in U}
\Big\{\dfrac{1}{2}tr(\sigma\sigma^{T}(t, x, u, v) A)+ p^{T} b(t, x,
u, v)+ f_{2}(t, x, y,  p^{T}\sigma(t, x, u, v), u, v) \Big\}.
 \end{aligned}
\end{eqnarray*}

 Under the above condition, from the above section we see that:   for $(t,x)\in[0,T]\times\mathbb{R}^{n}$,
\begin{eqnarray}\label{equa}
&&W_{1}(t,x)= \esssup_{\alpha\in\mathcal {A}_{t,T}}
\essinf_{\beta\in\mathcal {B}_{t,T}}
J_{1}(t,x;\alpha,\beta)=\essinf_{\beta\in\mathcal
{B}_{t,T}}\esssup_{\alpha\in\mathcal {A}_{t,T}}
J_{1}(t,x;\alpha,\beta), \nonumber\\
&& W_{2}(t,x)=\essinf_{\alpha\in\mathcal
{A}_{t,T}}\esssup_{\beta\in\mathcal {B}_{t,T}}
J_{2}(t,x;\alpha,\beta)=\esssup_{\beta\in\mathcal
{B}_{t,T}}\essinf_{\alpha\in\mathcal {A}_{t,T}}
J_{2}(t,x;\alpha,\beta).
\end{eqnarray}
In order to simplify arguments, let us  also assume that the
coefficients $b, \sigma, f_{j}, \Phi_{j}, f_{j}$ and  $ h_{j}$
($j=1,2$), satisfy the assumptions (H3.1)--(H3.4) and  are bounded.

We recall the definition of the Nash equilibrium payoff of
nonzero-sum stochastic differential games, which was introduced in
Buckdahn, Cardaliaguet and Rainer \cite{BCR2004}  and Lin
\cite{L2011}.
\begin{definition}\label{d2}
A couple $(e_{1},e_{2})\in\mathbb{R}^{2}$ is called a Nash
equilibrium payoff at the point $(t,x)$ if for any $\varepsilon>0$,
there exists $(\alpha_{\varepsilon},\beta_{\varepsilon})\in \mathcal
{A}_{t,T}\times \mathcal {B}_{t,T}$ such that, for all
$(\alpha,\beta)\in \mathcal {A}_{t,T}\times \mathcal {B}_{t,T},$
\begin{eqnarray}\label{eq4}
J_{1}(t,x;\alpha_{\varepsilon},\beta_{\varepsilon})\geq
J_{1}(t,x;\alpha,\beta_{\varepsilon})-\varepsilon,\
J_{2}(t,x;\alpha_{\varepsilon},\beta_{\varepsilon})\geq
J_{2}(t,x;\alpha_{\varepsilon},\beta)-\varepsilon,\ \mathbb{P}-a.s.,
\end{eqnarray}
and
\begin{eqnarray*}
|\mathbb{E}[J_{j}(t,x;\alpha_{\varepsilon},\beta_{\varepsilon})]-e_{j}|\leq
\varepsilon, \ j=1,2.
\end{eqnarray*}
\end{definition}

From  Lemma \ref{l2} it follows that the following lemma holds.
\begin{lemma}\label{le3}
For any $\varepsilon>0$ and
$(\alpha_{\varepsilon},\beta_{\varepsilon})\in \mathcal
{A}_{t,T}\times \mathcal {B}_{t,T}$, (\ref{eq4}) holds if and only
if,  for all $(u,v)\in \mathcal {U}_{t,T}\times \mathcal {V}_{t,T}$,
\begin{eqnarray}\label{eq5}
J_{1}(t,x;\alpha_{\varepsilon},\beta_{\varepsilon})\geq
J_{1}(t,x;u,\beta_{\varepsilon}(u))-\varepsilon,\
J_{2}(t,x;\alpha_{\varepsilon},\beta_{\varepsilon})\geq
J_{2}(t,x;\alpha_{\varepsilon}(v),v)-\varepsilon, \ \mathbb{P}-a.s.
\end{eqnarray}
\end{lemma}

Before giving the main results in this sections we first introduce
the following lemma.
\begin{lemma}\label{le1}
Let $(t,x)\in [0,T]\times \mathbb{R}^{n}$ and $u\in \mathcal
{U}_{t,T}$ be arbitrarily fixed. Then,

(i) for all $\delta\in[0,T-t]$ and $\varepsilon>0$,  there exists an
NAD strategy $\alpha\in \mathcal {A}_{t,T}$ such that, for all
 $v\in \mathcal {V}_{t,T}$,
\begin{eqnarray*}\label{}
\alpha(v)&=&u, \text{on}\ [t,t+\delta],\\
^{2}Y^{t,x;\alpha(v),v}_{t+\delta}&\leq&
W_{2}(t+\delta,X^{t,x;\alpha(v),v}_{t+\delta})+\varepsilon,\
\mathbb{P}-a.s.,
\end{eqnarray*}

(ii) for all $\delta\in[0,T-t]$ and $\varepsilon>0$,  there exists
an NAD strategy  $\alpha\in \mathcal {A}_{t,T}$ such that, for all
 $v\in \mathcal {V}_{t,T}$,
\begin{eqnarray*}\label{}
\alpha(v)&=&u, \text{on}\ [t,t+\delta],\\
^{1}Y^{t,x;\alpha(v),v}_{t+\delta}&\geq&
W_{1}(t+\delta,X^{t,x;\alpha(v),v}_{t+\delta})-\varepsilon, \
\mathbb{P}-a.s.
\end{eqnarray*}
\end{lemma}
Using arguments similar to Lin \cite{L2011} we can prove this lemma.
The proof is omitted here.  We also need the following lemma, which
can be established by  standard arguments for SDEs.
\begin{lemma}\label{le2}
There exists a positive constant  $C$ such that, for all
$(u,v),(u',v')\in \mathcal {U}_{t,T}\times \mathcal {V}_{t,T}$, and
for all $\mathcal {F}_{r}$-stopping times $S:\Omega\rightarrow[t,T]$
with $X^{t,x;u,v}_{S}=X^{t,x;u',v'}_{S}$, $\mathbb{P}-a.s.,$ it
holds,  for all real $\tau\in[0,T],$
\begin{eqnarray*}\label{}
\mathbb{E}[\sup\limits_{0\leq s\leq \tau}|X^{t,x;u,v}_{(S+s)\wedge
T}-X^{t,x;u',v'}_{(S+s)\wedge T}|^{2} \Big|\mathcal {F}_{t}]\leq
C\tau, \ \mathbb{P}-a.s.
\end{eqnarray*}
\end{lemma}

Let us now give one of main results in this section:  the
characterization  of Nash equilibrium payoffs for nonzero-sum
stochastic differential games with reflection. We postpone its proof
to Section 6.
\begin{theorem}\label{t1}
Let  Isaacs condition (\ref{Isaacs}) hold and $(t,x)\in [0,T]\times
 \mathbb{R}^{n}$.  If for all $\varepsilon>0$, there exist
 $u^{\varepsilon}\in \mathcal {U}_{t,T}$ and $v^{\varepsilon}\in \mathcal
 {V}_{t,T}$ such that for all $s\in[t,T]$ and $j=1,2,$\\
\begin{eqnarray}\label{eq6}
\mathbb{P}\Big(\ ^{j}Y^{t,x;u^{\varepsilon},v^{\varepsilon}}_{s}\geq
W_{j}(s,X^{t,x;u^{\varepsilon},v^{\varepsilon}}_{s})-\varepsilon\ |\
\mathcal {F}_{t}\Big)\geq 1-\varepsilon,\ \mathbb{P}-a.s.,
\end{eqnarray}
and
\begin{eqnarray}\label{eq7}
|\mathbb{E}[J_{j}(t,x;u^{\varepsilon},v^{\varepsilon})]-
e_{j}|\leq\varepsilon,
\end{eqnarray}
then $(e_{1},e_{2})\in\mathbb{R}^{2}$ is a Nash equilibrium payoff
at point $(t,x)$.
\end{theorem}\vskip2mm

Before giving   the existence theorem of a Nash equilibrium payoff
we first establish the following proposition, which is crucial for
the proof of the existence theorem of a Nash equilibrium payoff.
\begin{proposition}\label{p2}
Under the assumptions of Theorem \ref{t1},  for all $\varepsilon>0,$
there exists
$(u^{\varepsilon},v^{\varepsilon})\in\mathcal{U}_{t,T}\times\mathcal{V}_{t,T}$
independent of $\mathcal {F}_{t}$ such that, for all $t\leq
s_{1}\leq s_{2}\leq T$, $j=1,2$,
\begin{eqnarray*}\label{}
\mathbb{P}\Big(\
 W_{j}(s_{1},X^{t,x;u^{\varepsilon},v^{\varepsilon}}_{s_{1}})-\varepsilon\leq
\ ^{j}G^{t,x;u^{\varepsilon},v^{\varepsilon}}_{s_{1},s_{2}}
 [W_{j}(s_{2},X^{t,x;u^{\varepsilon},v^{\varepsilon}}_{s_{2}})]\Big |\mathcal {F}_{t}\Big)> 1- \varepsilon.
\end{eqnarray*}
\end{proposition}

Let us first give some preliminary result. Since the   proof of the
following lemma is similar to that in \cite{L2011}, we omit here.
\begin{lemma}
For all $\varepsilon>0,$  all $\delta\in[0,T-t]$ and
$x\in\mathbb{R}^{n}$, there exists
$(u^{\varepsilon},v^{\varepsilon})\in\mathcal{U}_{t,T}\times\mathcal{V}_{t,T}$
independent of $\mathcal {F}_{t}$, such that, $j=1,2,$
\begin{eqnarray*}\label{}
 W_{j}(t,x)-\varepsilon\leq
 \ ^{j}G^{t,x;u^{\varepsilon},v^{\varepsilon}}_{t,t+\delta}
 [W_{j}(t+\delta,X^{t,x;u^{\varepsilon},v^{\varepsilon}}_{t+\delta})], \ \mathbb{P}- a.s.
\end{eqnarray*}
\end{lemma}

Let us establish the following Lemma.
\begin{lemma}\label{le6}
Let $n\geq 1$ and let us fix some partition
$t=t_{0}<t_{1}<\cdots<t_{n}=T$ of the interval $[t,T]$. Then, for
all $\varepsilon>0,$  there exists
$(u^{\varepsilon},v^{\varepsilon})\in\mathcal{U}_{t,T}\times\mathcal{V}_{t,T}$
independent of $\mathcal {F}_{t}$, such that, for all
$i=0,\cdots,n-1$,
\begin{eqnarray*}\label{}
 W_{j}(t_{i},X^{t,x;u^{\varepsilon},v^{\varepsilon}}_{t_{i}})-\varepsilon\leq
\ ^{j}G^{t,x;u^{\varepsilon},v^{\varepsilon}}_{t_{i},t_{i+1}}
 [W_{j}(t_{i+1},X^{t,x;u^{\varepsilon},v^{\varepsilon}}_{t_{i+1}})],\
 \mathbb{P}- a.s.
\end{eqnarray*}
\end{lemma}

\begin{proof}
Let us prove this lemma by induction. By the above lemma, it is
obvious for $i=0$. We now suppose that
$(u^{\varepsilon},v^{\varepsilon})$ independent of $\mathcal
{F}_{t}$, is constructed on the interval $[t,t_{i})$ and we shall
define  it on $[t_{i},t_{i+1})$. From the above lemma it follows
that, for all  $y\in \mathbb{R}^{n}$, there exists
$(u^{y},v^{y})\in\mathcal{U}_{t_{i},T}\times\mathcal{V}_{t_{i},T}$
independent of $\mathcal {F}_{t}$, such that,
\begin{eqnarray}\label{eq10}
 W_{j}(t_{i},y)-\frac{\varepsilon}{2}\leq
\ ^{j}G^{t_{i},y;u^{y},v^{y}}_{t_{i},t_{i+1}}
 [W_{j}(t_{i+1},X^{t,y;u^{y},v^{y}}_{t_{i+1}})], \ \mathbb{P}- a.s,
 j=1,2.
\end{eqnarray}
For arbitrarily $j=1,2,$   for all $y, z\in \mathbb{R}^{n}$ and
$s\in[t_{i},t_{i+1}]$, we set
$$y^{1}_{s}=\ ^{j}G^{t_{i},y;u^{y},v^{y}}_{s,t_{i+1}}
 [W_{j}(t_{i+1},X^{t_{i},y;u^{y},v^{y}}_{t_{i+1}})], \ \text{and}\ \
 y^{2}_{s}=\ ^{j}G^{t_{i},z;u^{y},v^{y}}_{s,t_{i+1}}
 [W_{j}(t_{i+1},X^{t_{i},z;u^{y},v^{y}}_{t_{i+1}})].$$
Then let us consider the following  BSDEs:
\begin{eqnarray*}\label{}
 \left \{\begin{array}{rl}
    &y^{1}_{s}= W_{j}(t_{i+1},X^{t_{i},y;u^{y},v^{y}}_{t_{i+1}})
+\displaystyle \int_{s}^{t_{i+1}} f_{j}(r,X^{t_{i},y;u^{y},v^{y}}_r,
y_r^{1},z^{1}_r,u^{y}_r, v^{y}_r)dr \\& \hskip2cm  +\
^{1}K_{t_{i+1}}- \ ^{1}K_s-\displaystyle
\int_{s}^{t_{i+1}} z^{1}_rdB_r,  \\
      & y^{1}_{s} \geq  h_{j}(s,X^{t_{i},y;u^{y},v^{y}}_s),  \qquad s\in [t_{i},t_{i+1}],\\
     &^{1}K_{t_{i}}=0,\quad  \displaystyle\int_{t_{i}}^{{t_{i+1}}}
     (y^{1}_{s} - h_{j}(s,X^{t_{i},y;u^{y},v^{y}}_s))d\ ^{1}K_{r}=0,
     \end{array}\right.
 \end{eqnarray*}
and
\begin{eqnarray*}\label{}
 \left \{\begin{array}{rl}
    &y^{2}_{s}= W_{j}(t_{i+1},X^{t_{i},z;u^{y},v^{y}}_{t_{i+1}})
+\displaystyle \int_{s}^{t_{i+1}} f_{j}(r,X^{t_{i},z;u^{y},v^{y}}_r,
y_r^{2},z^{2}_r,u^{y}_r, v^{y}_r)dr \\& \hskip2cm  +\
^{2}K_{t_{i+1}}- \ ^{2}K_s-\displaystyle
\int_{s}^{t_{i+1}} z^{2}_rdB_r,  \\
      & y^{2}_{s} \geq  h_{j}(s,X^{t_{i},z;u^{y},v^{y}}_s),  \qquad s\in [t_{i},t_{i+1}],\\
     &^{2}K_{t_{i}}=0,\quad  \displaystyle\int_{t_{i}}^{{t_{i+1}}}
     (y^{2}_{s} - h_{j}(s,X^{t_{i},z;u^{y},v^{y}}_s))d\ ^{2}K_{s}=0.
     \end{array}\right.
 \end{eqnarray*}
From the Lemmas \ref{l18} and \ref{l4} it follows that
\begin{eqnarray*}\label{}
&&|^{j}G^{t_{i},y;u^{y},v^{y}}_{t_{i},t_{i+1}}
 [W_{j}(t_{i+1},X^{t,y;u^{y},v^{y}}_{t_{i+1}})]
-\ ^{j}G^{t_{i},z;u^{y},v^{y}}_{t_{i},t_{i+1}}
 [W_{j}(t_{i+1},X^{t_{i},z;u^{y},v^{y}}_{t_{i+1}})]|^{2}\\
&\leq&C\mathbb{E}[|W_{j}(t_{i+1},X^{t_{i},y;u^{y},v^{y}}_{t_{i+1}})
-W_{j}(t_{i+1},X^{t_{i},z;u^{y},v^{y}}_{t_{i+1}})|^{2}
\Big|\mathcal {F}_{t_{i}}]\\
&&+C\mathbb{E}[|\displaystyle \int_{t_{i}}^{t_{i+1}}
f_{j}(r,X^{t_{i},y;u^{y},v^{y}}_r, y_r^{1},z^{1}_r,u^{y}_r,
v^{y}_r)dr -\displaystyle \int_{t_{i}}^{t_{i+1}}
f_{j}(r,X^{t_{i},z;u^{y},v^{y}}_r, y_r^{1},z^{1}_r,u^{y}_r,
v^{y}_r)dr|^{2}\Big|\mathcal {F}_{t_{i}}]\\
&&+C\mathbb{E}[\sup\limits_{t_{i}\leq s\leq t_{i+1}}|
h_{j}(s,X^{t_{i},y;u^{y},v^{y}}_s) -
h_{j}(s,X^{t_{i},z;u^{y},v^{y}}_s)|^{2}
\Big|\mathcal {F}_{t_{i}}]^{\frac{1}{2}}\\
&\leq&C\mathbb{E}[|X^{t_{i},y;u^{y},v^{y}}_{t_{i+1}}
-X^{t_{i},z;u^{y},v^{y}}_{t_{i+1}}|^{2} \Big|\mathcal
{F}_{t_{i}}]+C\mathbb{E}[\displaystyle \int_{t_{i}}^{t_{i+1}}|
X^{t_{i},y;u^{y},v^{y}}_r-X^{t_{i},z;u^{y},v^{y}}_r |^{2}dr
\Big|\mathcal {F}_{t_{i}}]\\
&&+C\mathbb{E}[\sup\limits_{t_{i}\leq s\leq t_{i+1}}|
X^{t_{i},y;u^{y},v^{y}}_s - X^{t_{i},z;u^{y},v^{y}}_s|^{2}
\Big|\mathcal {F}_{t_{i}}]^{\frac{1}{2}}\\
 &\leq & C|y-z|.
\end{eqnarray*}
Combining  the above inequality, Proposition \ref{pro} and
(\ref{eq10}) we see that
\begin{eqnarray*}\label{}
 W_{j}(t_{i},z)-\varepsilon&\leq&
 W_{j}(t_{i},y)-\varepsilon+C|y-z|^{\frac{1}{2}}\\
 &\leq& \ ^{j}G^{t_{i},y;u^{y},v^{y}}_{t_{i},t_{i+1}}
 [W_{j}(t_{i+1},X^{t,y;u^{y},v^{y}}_{t_{i+1}})]-\frac{\varepsilon}{2}+C|y-z|^{\frac{1}{2}}\\
&\leq& \ ^{j}G^{t,z;u^{y},v^{y}}_{t_{i},t_{i+1}}
 [W_{j}(t_{i+1},X^{t,z;u^{y},v^{y}}_{t_{i+1}})]-\frac{\varepsilon}{2}+C|y-z|^{\frac{1}{2}}\\
 &\leq& \ ^{j}G^{t,z;u^{y},v^{y}}_{t_{i},t_{i+1}}
 [W_{j}(t_{i+1},X^{t,z;u^{y},v^{y}}_{t_{i+1}})], \ \mathbb{P}- a.s.,
\end{eqnarray*}
for $C|y-z|^{\frac{1}{2}}\leq \dfrac{\varepsilon}{2}$.\\

 Let $\{O_{i}\}_{i\geq1}\subset \mathcal {B}(\mathbb{R}^{n})$ be a partition of
$\mathbb{R}^{n}$ with $ diam(O_{i})<\dfrac{\varepsilon}{2C}$ and let
$y_{l}\in O_{l}$. Then, for $z\in O_{l}$,
\begin{eqnarray}\label{e6}
 W_{j}(t_{i},z)-\varepsilon\leq
\ ^{j}G^{t,z;u^{y_{l}},v^{y_{l}}}_{t_{i},t_{i+1}}
 [W_{j}(t_{i+1},X^{t,z;u^{y_{l}},v^{y_{l}}}_{t_{i+1}})], \ \mathbb{P}- a.s.
\end{eqnarray}
Let us put
\begin{eqnarray*}\label{}
u^{\varepsilon}=\sum\limits_{l\geq1}1_{O_{l}}(X^{t,x;u^{\varepsilon},v^{\varepsilon}})u^{y_{l}},\
v^{\varepsilon}=\sum\limits_{l\geq1}1_{O_{l}}(X^{t,x;u^{\varepsilon},v^{\varepsilon}})v^{y_{l}}.
\end{eqnarray*}
Then
\begin{eqnarray*}\label{}
&& ^{j}G^{t,x;u^{\varepsilon},v^{\varepsilon}}_{t_{i},t_{i+1}}
 [W_{j}(t_{i+1},X^{t,x;u^{\varepsilon},v^{\varepsilon}}_{t_{i+1}})]\\
  &=& ^{j}G^{t_{i},X^{t,x;u^{\varepsilon},v^{\varepsilon}}_{t_{i}};u^{\varepsilon},v^{\varepsilon}}_{t_{i},t_{i+1}}
 [\sum\limits_{l\geq1}W_{j}(t_{i+1},X^{t_{i},X^{t,x;u^{\varepsilon},v^{\varepsilon}}_{t_{i}};
 u^{\varepsilon},v^{\varepsilon}}_{t_{i+1}})1_{O_{l}}(X^{t,x;u^{\varepsilon},v^{\varepsilon}}_{t_{i}})]\\
   &=& ^{j}G^{t_{i},X^{t,x;u^{\varepsilon},v^{\varepsilon}}_{t_{i}};u^{\varepsilon},v^{\varepsilon}}_{t_{i},t_{i+1}}
 [\sum\limits_{l\geq1}W_{j}(t_{i+1},X^{t_{i},X^{t,x;u^{\varepsilon},v^{\varepsilon}}_{t_{i}};
 u^{y_{l}},v^{y_{l}}}_{t_{i+1}})1_{O_{l}}(X^{t,x;u^{\varepsilon},v^{\varepsilon}}_{t_{i}})]\\
    &=& \sum\limits_{l\geq1}\ ^{j}G^{t_{i},X^{t,x;u^{\varepsilon},v^{\varepsilon}}_{t_{i}};
    u^{y_{l}},v^{y_{l}}}_{t_{i},t_{i+1}}
 [W_{j}(t_{i+1},X^{t_{i},X^{t,x;u^{\varepsilon},v^{\varepsilon}}_{t_{i}};
 u^{y_{l}},v^{y_{l}}}_{t_{i+1}})]1_{O_{l}}(X^{t,x;u^{\varepsilon},v^{\varepsilon}}_{t_{i}}),
\end{eqnarray*}
which together with  (\ref{e6}) yields
\begin{eqnarray*}\label{}
&& ^{j}G^{t,x;u^{\varepsilon},v^{\varepsilon}}_{t_{i},t_{i+1}}
 [W_{j}(t_{i+1},X^{t,x;u^{\varepsilon},v^{\varepsilon}}_{t_{i+1}})]
 \geq  \sum\limits_{l\geq1}
[W_{j}(t_{i},X^{t,x;u^{y_{l}},v^{y_{l}}}_{t_{i}})-\varepsilon]
1_{O_{l}}(X^{t,x;u^{\varepsilon},v^{\varepsilon}}_{t_{i}})\\
&=& \sum\limits_{l\geq1}
W_{j}(t_{i},X^{t,x;u^{y_{l}},v^{y_{l}}}_{t_{i}})
1_{O_{l}}(X^{t,x;u^{\varepsilon},v^{\varepsilon}}_{t_{i}})-\varepsilon
  =   W_{j}(t_{i},X^{t,x;u^{\varepsilon},v^{\varepsilon}}_{t_{i}})-\varepsilon,
\end{eqnarray*}
from which we get the desired result.
\end{proof}\vskip2mm

\noindent  Let us come to the proof of Proposition \ref{p2}.
\begin{proof}
Let $t=t_{0}<t_{1}<\cdots<t_{n}=T$ be a partition of  $[t,T]$, and
  $\tau=\sup\limits_{i}(t_{i+1}-t_{i})$.  By Proposition \ref{pro} and Lemma
\ref{le2} we see that, for all $j=1,2$, $0\leq k\leq n$,
$s\in[t_{k},t_{k+1})$ and $(u,v)\in\mathcal{U}_{t
,T}\times\mathcal{V}_{t,T}$,
\begin{eqnarray}\label{e8}
&&\mathbb{E}[|W_{j}(t_{k},X_{t_{k}}^{t,x;u,v})
-W_{j}(s,X^{t,x;u,v}_{s})|^{2}]\nonumber\\
 &\leq&2\mathbb{E}[|W_{j}(t_{k},X_{t_{k}}^{t,x;u,v})
-W_{j}(s,X^{t,x;u,v}_{t_{k}})|^{2}]\nonumber\\&&+2\mathbb{E}[|W_{j}(s,X_{t_{k}}^{t,x;u,v})
-W_{j}(s,X^{t,x;u,v}_{s})|^{2}]\nonumber\\
&\leq&C|s-t_{k}| (1+
\mathbb{E}[|X_{t_{k}}^{t,x;u,v}|^{2}])+C\mathbb{E}[|X_{t_{k}}^{t,x;u,v}
-X^{t,x;u,v}_{s}|^{2}]\nonumber\\
 &\leq& C\tau.
\end{eqnarray}
Here and after $C$ represents  a generic constant which may be
different at different places.

We  let $(u^{\varepsilon},v^{\varepsilon})\in\mathcal{U}_{t
,T}\times\mathcal{V}_{t,T}$ be defined as in  Lemma \ref{le6} for
$\varepsilon=\varepsilon_{0}$, where $\varepsilon_{0}>0$ will be
specified  later. Then,  for all $i, 0\leq i \leq n$,
\begin{eqnarray*}\label{}
 W_{j}(t_{i},X^{t,x;u^{\varepsilon},v^{\varepsilon}}_{t_{i}})-\varepsilon_{0}\leq
 \ ^{j}G^{t,x;u^{\varepsilon},v^{\varepsilon}}_{t_{i},t_{i+1}}
 [W_{j}(t_{i+1},X^{t,x;u^{\varepsilon},v^{\varepsilon}}_{t_{i+1}})],\
 \mathbb{P}- a.s.
\end{eqnarray*}
For  $t\leq s_{1}\leq s_{2}\leq T$, let us  suppose, without loss of
generality, that $t_{i-1}\leq s_{1} < t_{i}$ and $t_{k}< s_{2} \leq
t_{k+1}$, for some $1\leq i<k\leq n-1$. Therefore, applying the
Lemmas \ref{l18} and \ref{l8} we deduce that
\begin{eqnarray*}\label{}
&& ^{j}G^{t,x;u^{\varepsilon},v^{\varepsilon}}_{t_{i},t_{k+1}}
 [W_{j}(t_{k+1},X^{t,x;u^{\varepsilon},v^{\varepsilon}}_{t_{k+1}})]\\
 &=& ^{j}G^{t,x;u^{\varepsilon},v^{\varepsilon}}_{t_{i},t_{k}}[^{j}G^{t,x;u^{\varepsilon},v^{\varepsilon}}_{t_{k},t_{k+1}}
 [W_{j}(t_{k+1},X^{t,x;u^{\varepsilon},v^{\varepsilon}}_{t_{k+1}})]]\\
  &\geq& ^{j}G^{t,x;u^{\varepsilon},v^{\varepsilon}}_{t_{i},t_{k}}
  [W_{j}(t_{k},X^{t,x;u^{\varepsilon},v^{\varepsilon}}_{t_{k}})-\varepsilon_{0}]\\
   &\geq& ^{j}G^{t,x;u^{\varepsilon},v^{\varepsilon}}_{t_{i},t_{k}}
  [W_{j}(t_{k},X^{t,x;u^{\varepsilon},v^{\varepsilon}}_{t_{k}})]-C\varepsilon_{0}\\
   &\geq&\cdots\geq \ ^{j}G^{t,x;u^{\varepsilon},v^{\varepsilon}}_{t_{i},t_{i+1}}
  [W_{j}(t_{i+1},X^{t,x;u^{\varepsilon},v^{\varepsilon}}_{t_{i+1}})]-C(k-i)\varepsilon_{0}\\
 &\geq&
 W_{j}(t_{i},X^{t,x;u^{\varepsilon},v^{\varepsilon}}_{t_{i}})-C(k-i+1)\varepsilon_{0},
\end{eqnarray*}
from which we see that
\begin{eqnarray*}\label{}
&& ^{j}G^{t,x;u^{\varepsilon},v^{\varepsilon}}_{s_{1},t_{k+1}}
 [W_{j}(t_{k+1},X^{t,x;u^{\varepsilon},v^{\varepsilon}}_{t_{k+1}})]\\
 &=& ^{j}G^{t,x;u^{\varepsilon},v^{\varepsilon}}_{s_{1},t_{i}}
 [\ ^{j}G^{t,x;u^{\varepsilon},v^{\varepsilon}}_{t_{i},t_{k+1}}
 [W_{j}(t_{k+1},X^{t,x;u^{\varepsilon},v^{\varepsilon}}_{t_{k+1}})]]\\
 &\geq& ^{j}G^{t,x;u^{\varepsilon},v^{\varepsilon}}_{s_{1},t_{i}}
  [W_{j}(t_{i},X^{t,x;u^{\varepsilon},v^{\varepsilon}}_{t_{i}})-C(k-i+1)\varepsilon_{0}]\\
&\geq& ^{j}G^{t,x;u^{\varepsilon},v^{\varepsilon}}_{s_{1},t_{i}}
  [W_{j}(t_{i},X^{t,x;u^{\varepsilon},v^{\varepsilon}}_{t_{i}})]-C(k-i+2)\varepsilon_{0}\\
  &\geq& ^{j}G^{t,x;u^{\varepsilon},v^{\varepsilon}}_{s_{1},t_{i}}
  [W_{j}(t_{i},X^{t,x;u^{\varepsilon},v^{\varepsilon}}_{t_{i}})]-\frac{\varepsilon}{2},
\end{eqnarray*}
where we put $\varepsilon_{0}=\dfrac{\varepsilon}{2Cn}$. We set
\begin{eqnarray}\label{eq110}
 I_{1}&=& ^{j}G^{t,x;u^{\varepsilon},v^{\varepsilon}}_{s_{1},t_{k+1}}
 [W_{j}(t_{k+1},X^{t,x;u^{\varepsilon},v^{\varepsilon}}_{t_{k+1}})]-\
   ^{j}G^{t,x;u^{\varepsilon},v^{\varepsilon}}_{s_{1},t_{i}}
  [W_{j}(t_{i},X^{t,x;u^{\varepsilon},v^{\varepsilon}}_{t_{i}})]+\frac{\varepsilon}{2}\geq0,\nonumber\\
 I_{2}&=&  ^{j}G^{t,x;u^{\varepsilon},v^{\varepsilon}}_{s_{1},s_{2}}
 [W_{j}(s_{2},X^{t,x;u^{\varepsilon},v^{\varepsilon}}_{s_{2}})]
 -W_{j}(s_{1},X^{t,x;u^{\varepsilon},v^{\varepsilon}}_{s_{1}})+\frac{\varepsilon}{2}.
\end{eqnarray}
In what follows we shall prove the following:
\begin{eqnarray*}\label{}
\mathbb{E}[| I_{1} - I_{2}|^{2}] \leq C\tau.
\end{eqnarray*}
Let us put
$$y_{s}=\
^{j}G^{t,x;u^{\varepsilon},v^{\varepsilon}}_{s,t_{i}}
  [W_{j}(t_{i},X^{t,x;u^{\varepsilon},v^{\varepsilon}}_{t_{i}})], s\in [s_{1},
  t_{i}].$$
 Then we consider  the associated BSDEs:
\begin{eqnarray*}\label{}
 \left \{\begin{array}{rl}
    &y_{s}=W_{j}(t_{i},X^{t,x;u^{\varepsilon},v^{\varepsilon}}_{t_{i}})
+\displaystyle \int_{s}^{t_{i}}
f_{j}(r,X^{t,x;u^{\varepsilon},v^{\varepsilon}}_r,
y_r,z_r,u^{\varepsilon}_r,v^{\varepsilon}_r)dr
+k_{t_{i}}-k_{s}-\displaystyle \int_{s}^{t_{i}}z_rdB_r,  \\
      & y_{s} \geq  h_{j}(s,X^{t,x;u^{\varepsilon},v^{\varepsilon}}_s),  \qquad s\in [s_{1}, t_{i}],\\
     &k_{s_{1}}=0,\quad  \displaystyle\int_{s_{1}}^{{t_{i}}}
     (y_{r} - h_{j}(r,X^{t,x;u^{\varepsilon},v^{\varepsilon}}_r))d k_{r}=0,
     \end{array}\right.
 \end{eqnarray*}
and
\begin{equation*}\label{}
\begin{array}{lll}
y'_{s}=
W_{j}(s_{1},X^{t,x;u^{\varepsilon},v^{\varepsilon}}_{s_{1}}), \ s\in
[s_{1}, t_{i}].
\end{array}
\end{equation*}
Thus,  applying  Lemma \ref{l18} we conclude
\begin{eqnarray*}
&&|^{j}G^{t,x;u^{\varepsilon},v^{\varepsilon}}_{s_{1},t_{i}}
  [W_{j}(t_{i},X^{t,x;u^{\varepsilon},v^{\varepsilon}}_{t_{i}})]
  -W_{j}(s_{1},X^{t,x;u^{\varepsilon},v^{\varepsilon}}_{s_{1}})|^{2}\\
&\leq &
C\mathbb{E}[|W_{j}(t_{i},X^{t,x;u^{\varepsilon},v^{\varepsilon}}_{t_{i}})-
W_{j}(s_{1},X^{t,x;u^{\varepsilon},v^{\varepsilon}}_{s_{1}})
|^2|\mathcal {F}_{s_{1}}] \\
& & + C \mathbb{E}[\int_{s_{1}}^{t_{i}}
|f_{j}(r,X^{t,x;u^{\varepsilon},v^{\varepsilon}}_r,
y_r,z_r,u^{\varepsilon}_r, v^{\varepsilon}_r) |^2|\mathcal
{F}_{s_{1}}]\\
&&+C\mathbb{E}[\sup\limits_{s_{1}\leq s\leq t_{i}}|
h_{j}(s,X^{t,x;u^{\varepsilon},v^{\varepsilon}}_s) -
h_{j}(s_{1},X^{t,x;u^{\varepsilon},v^{\varepsilon}}_{s_{1}})|^{2}dr
\Big|\mathcal {F}_{s_{1}}]^{\frac{1}{2}}\\
&\leq &
C\mathbb{E}[|W_{j}(t_{i},X^{t,x;u^{\varepsilon},v^{\varepsilon}}_{t_{i}})-
W_{j}(s_{1},X^{t,x;u^{\varepsilon},v^{\varepsilon}}_{s_{1}})
|^2|\mathcal {F}_{s_{1}}] \\
& & +  C (t_{i}-s_{1})^{\alpha}+C\mathbb{E}[\sup\limits_{s_{1}\leq
s\leq t_{i}}| X^{t,x;u^{\varepsilon},v^{\varepsilon}}_s -
X^{t,x;u^{\varepsilon},v^{\varepsilon}}_{s_{1}}|^{2} \Big|\mathcal
{F}_{s_{1}}]^{\frac{1}{2}},
\end{eqnarray*}
where we have used the assumptions (H3.3) and (H3.4) and the
boundedness of $f_{j}$. Since
$(u^{\varepsilon},v^{\varepsilon})\in\mathcal{U}_{t,T}\times\mathcal{V}_{t,T}$
is independent of $\mathcal {F}_{t}$ we have
\begin{eqnarray*}
&&\mathbb{E}[|^{j}G^{t,x;u^{\varepsilon},v^{\varepsilon}}_{s_{1},t_{i}}
  [W_{j}(t_{i},X^{t,x;u^{\varepsilon},v^{\varepsilon}}_{t_{i}})]
  -W_{j}(s_{1},X^{t,x;u^{\varepsilon},v^{\varepsilon}}_{s_{1}})|^{2}|\mathcal {F}_{t}]\\
  &\leq &
C\mathbb{E}[|W_{j}(t_{i},X^{t,x;u^{\varepsilon},v^{\varepsilon}}_{t_{i}})-
W_{j}(s_{1},X^{t,x;u^{\varepsilon},v^{\varepsilon}}_{s_{1}}) |^2] +
C (t_{i}-s_{1})\\
&&+C\mathbb{E}[\sup\limits_{s_{1}\leq s\leq t_{i}}|
X^{t,x;u^{\varepsilon},v^{\varepsilon}}_s -
X^{t,x;u^{\varepsilon},v^{\varepsilon}}_{s_{1}}|^{2} \Big|\mathcal
{F}_{t}]^{\frac{1}{2}}.
\end{eqnarray*}
By virtue of   (\ref{e8}) we have
\begin{eqnarray}\label{e10}
\mathbb{E}[|^{j}G^{t,x;u^{\varepsilon},v^{\varepsilon}}_{s_{1},t_{i}}
  [W_{j}(t_{i},X^{t,x;u^{\varepsilon},v^{\varepsilon}}_{t_{i}})]
  -W_{j}(s_{1},X^{t,x;u^{\varepsilon},v^{\varepsilon}}_{s_{1}})|^{2}]\leq C\tau^{\frac{1}{2}}.
\end{eqnarray}
By a similar argument
\begin{eqnarray}\label{e11}
\mathbb{E}[|^{j}G^{t,x;u^{\varepsilon},v^{\varepsilon}}_{s_{2},t_{k+1}}
 [W_{j}(t_{k+1},X^{t,x;u^{\varepsilon},v^{\varepsilon}}_{t_{k+1}})]
-W_{j}(s_{2},X^{t,x;u^{\varepsilon},v^{\varepsilon}}_{s_{2}})|^{2}]\leq
C\tau^{\frac{1}{2}}.
\end{eqnarray}
For $s\in [s_{1},s_{2}]$ we put $$y^{1}_{s}=\
^{j}G^{t,x;u^{\varepsilon},v^{\varepsilon}}_{s,t_{k+1}}
 [W_{j}(t_{k+1},X^{t,x;u^{\varepsilon},v^{\varepsilon}}_{t_{k+1}})]
 =\ ^{j}G^{t,x;u^{\varepsilon},v^{\varepsilon}}_{s,s_{2}}
 [^{j}G^{t,x;u^{\varepsilon},v^{\varepsilon}}_{s_{2},t_{k+1}}
 [W_{j}(t_{k+1},X^{t,x;u^{\varepsilon},v^{\varepsilon}}_{t_{k+1}})]],$$
 and
  $$y^{2}_{s}=\ ^{j}G^{t,x;u^{\varepsilon},v^{\varepsilon}}_{s,s_{2}}
 [W_{j}(s_{2},X^{t,x;u^{\varepsilon},v^{\varepsilon}}_{s_{2}})].$$
Let us consider the associated BSDEs:
\begin{eqnarray*}\label{}
 \left \{\begin{array}{rl}
    &y^{1}_{s}=\
^{j}G^{t,x;u^{\varepsilon},v^{\varepsilon}}_{s_{2},t_{k+1}}
 [W_{j}(t_{k+1},X^{t,x;u^{\varepsilon},v^{\varepsilon}}_{t_{k+1}})]
+\displaystyle \int_{s}^{s_{2}}
f_{j}(r,X^{t,x;u^{\varepsilon},v^{\varepsilon}}_r,
y_r^{1},z^{1}_r,u^{\varepsilon}_r, v^{\varepsilon}_r)dr\\&\hskip25mm
+k^{1}_{s_{2}}-k^{1}_{s}-\displaystyle
\int_{s}^{s_{2}} z^{1}_rdB_r,\\
      & y^{1}_{s} \geq  h_{j}(s,X^{t,x;u^{\varepsilon},v^{\varepsilon}}_s),  \qquad s\in [s_{1},s_{2}],\\
     &k^{1}_{s_{1}}=0,\quad  \displaystyle\int_{s_{1}}^{s_{2}}
     (y_{r} - h_{j}(r,X^{t,x;u^{\varepsilon},v^{\varepsilon}}_r))d k^{1}_{r}=0,
     \end{array}\right.
 \end{eqnarray*}
and
\begin{eqnarray*}\label{}
 \left \{\begin{array}{rl}
    &y^{2}_{s}=W_{j}(s_{2},X^{t,x;u^{\varepsilon},v^{\varepsilon}}_{s_{2}})
+\displaystyle \int_{s}^{s_{2}}
f_{j}(r,X^{t,x;u^{\varepsilon},v^{\varepsilon}}_r,
y_r^{2},z^{2}_r,u^{\varepsilon}_r,
v^{\varepsilon}_r)dr\\&\hskip25mm+k^{2}_{s_{2}}-k^{2}_{s}-\displaystyle
\int_{s}^{s_{2}} z^{2}_rdB_r,\\
      & y^{2}_{s} \geq  h_{j}(s,X^{t,x;u^{\varepsilon},v^{\varepsilon}}_s),  \qquad s\in [s_{1},s_{2}],\\
     &k^{2}_{s_{1}}=0,\quad  \displaystyle\int_{s_{1}}^{s_{2}}
     (y_{r} - h_{j}(r,X^{t,x;u^{\varepsilon},v^{\varepsilon}}_r))d
     k^{2}_{r}=0.
     \end{array}\right.
 \end{eqnarray*}
From the Lemmas \ref{l18} and \ref{l4} it follows that
\begin{eqnarray*}\label{}
&&|^{j}G^{t,x;u^{\varepsilon},v^{\varepsilon}}_{s_{1},t_{k+1}}
 [W_{j}(t_{k+1},X^{t,x;u^{\varepsilon},v^{\varepsilon}}_{t_{k+1}})]
-\ ^{j}G^{t,x;u^{\varepsilon},v^{\varepsilon}}_{s_{1},s_{2}}
 [W_{j}(s_{2},X^{t,x;u^{\varepsilon},v^{\varepsilon}}_{s_{2}})]|^{2}\\
&&\leq
C\mathbb{E}[|^{j}G^{t,x;u^{\varepsilon},v^{\varepsilon}}_{s_{2},t_{k+1}}
 [W_{j}(t_{k+1},X^{t,x;u^{\varepsilon},v^{\varepsilon}}_{t_{k+1}})]
-W_{j}(s_{2},X^{t,x;u^{\varepsilon},v^{\varepsilon}}_{s_{2}})|^{2}
\Big|\mathcal {F}_{s_{1}}].
\end{eqnarray*}
Hence, by (\ref{e11}) we see that
\begin{eqnarray*}\label{}
\mathbb{E}[|^{j}G^{t,x;u^{\varepsilon},v^{\varepsilon}}_{s_{1},t_{k+1}}
 [W_{j}(t_{k+1},X^{t,x;u^{\varepsilon},v^{\varepsilon}}_{t_{k+1}})]
-\ ^{j}G^{t,x;u^{\varepsilon},v^{\varepsilon}}_{s_{1},s_{2}}
 [W_{j}(s_{2},X^{t,x;u^{\varepsilon},v^{\varepsilon}}_{s_{2}})]|^{2}]\leq C\tau^{\frac{1}{2}}.
\end{eqnarray*}
The above inequality and (\ref{e10}) yield
\begin{eqnarray*}\label{}
\mathbb{E}[| I_{1} - I_{2}|^{2}] \leq C\tau^{\frac{1}{2}}.
\end{eqnarray*}
Therefore,
\begin{eqnarray*}\label{}
\mathbb{P}(I_{2}\leq - \frac{\varepsilon}{2}) \leq \mathbb{P}( |
I_{1} - I_{2}|\geq \frac{\varepsilon}{2}) \leq \frac{4\mathbb{E}[|
I_{1} - I_{2}|^{2}]}{\varepsilon^{2}}\leq
\frac{4C\tau^{\frac{1}{2}}}{\varepsilon^{2}}\leq\varepsilon,
\end{eqnarray*}
where we choose $\tau\leq\Big(\dfrac{\varepsilon^{3}}{4C}\Big)^{2}$,
and by (\ref{eq110}) we have
\begin{eqnarray*}\label{}
\mathbb{P}\Big(\
 W_{j}(s_{1},X^{t,x;u^{\varepsilon},v^{\varepsilon}}_{s_{1}})-\varepsilon\leq
\ ^{j}G^{t,x;u^{\varepsilon},v^{\varepsilon}}_{s_{1},s_{2}}
 [W_{j}(s_{2},X^{t,x;u^{\varepsilon},v^{\varepsilon}}_{s_{2}})]\
\Big)\geq 1- \varepsilon.
\end{eqnarray*}
 We also refer to the fact that since
$(u^{\varepsilon},v^{\varepsilon})$ is independent of $\mathcal
{F}_{t}$, the conditional probability $\mathbb{P}(\cdot |\mathcal
{F}_{t})$ of the event
$\Big\{W_{j}(s_{1},X^{t,x;u^{\varepsilon},v^{\varepsilon}}_{s_{1}})-\varepsilon\leq
\ ^{j}G^{t,x;u^{\varepsilon},v^{\varepsilon}}_{s_{1},s_{2}}
 [W_{j}(s_{2},X^{t,x;u^{\varepsilon},v^{\varepsilon}}_{s_{2}})]\Big\}$
coincides with its probability. Indeed, also
$\Big\{W_{j}(s_{1},X^{t,x;u^{\varepsilon},v^{\varepsilon}}_{s_{1}})-\varepsilon\leq
\ ^{j}G^{t,x;u^{\varepsilon},v^{\varepsilon}}_{s_{1},s_{2}}
 [W_{j}(s_{2},X^{t,x;u^{\varepsilon},v^{\varepsilon}}_{s_{2}})]\Big\}$ is independent of $\mathcal {F}_{t}$
 The proof is complete.
\end{proof}

Finally, we  give another main result: the existence theorem of a
Nash equilibrium payoff.
\begin{theorem}\label{t2}
Under  the Isaacs condition  A,  for all
$(t,x)\in[0,T]\times\mathbb{R}^{n}$, there exists a Nash equilibrium
payoff at $(t,x)$.
\end{theorem}

\begin{proof}\label{}
By Theorem \ref{t1} we only have to prove that, for all
$\varepsilon>0,$ there exists
$(u^{\varepsilon},v^{\varepsilon})\in\mathcal{U}_{t,T}\times\mathcal{V}_{t,T}$
which satisfies (\ref{eq6}) and (\ref{eq7}) for $s\in[t,T], j=1,2$.
For  $\varepsilon>0$, let us  consider
$(u^{\varepsilon},v^{\varepsilon})\in\mathcal{U}_{t,T}\times\mathcal{V}_{t,T}$
given by Proposition \ref{p2}, i.e., in particular,
$(u^{\varepsilon},v^{\varepsilon})$ is independent of $\mathcal
{F}_{t}$. Setting  $s_{1}=t$ and $s_{2}=T$ in Proposition \ref{p2},
we get (\ref{eq6}). Since $(u^{\varepsilon},v^{\varepsilon})$ is
independent of $\mathcal {F}_{t}$,
$J_{j}(t,x;u^{\varepsilon},v^{\varepsilon}), j=1,2,$ are
deterministic and
$\Big\{(J_{1}(t,x;u^{\varepsilon},v^{\varepsilon}),J_{2}(t,x;u^{\varepsilon},v^{\varepsilon})),
\varepsilon>0\Big\}$ is a bounded sequence. Therefore, we can choose
an accumulation point of this sequence, as $\varepsilon\rightarrow
0$. We denote this point by $(e_{1},e_{2})$. From Theorem \ref{t1}
we see that  $(e_{1},e_{2})$ is a Nash equilibrium payoff at
$(t,x)$.
 The proof is complete.
\end{proof}\vskip2mm

\section{Proof of Theorem \ref{t1}}

We now give the proof of Theorem \ref{t1}.
\begin{proof}
For arbitrarily fixed  $\varepsilon>0$ and some  $\varepsilon_{0}>0$
($\varepsilon_{0}$ depends on $\varepsilon$ and will be precise
later), let us assume  that
$(u^{\varepsilon_{0}},v^{\varepsilon_{0}})\in\mathcal
{U}_{t,T}\times \mathcal {V}_{t,T}$ satisfies (\ref{eq6}) and
(\ref{eq7}), i.e., for all $s\in[t,T]$ and $j=1,2,$
\begin{eqnarray}\label{e12}
\mathbb{P}\Big(\
^{j}Y^{t,x;u^{\varepsilon_{0}},v^{\varepsilon_{0}}}_{s}\geq
W_{j}(s,X^{t,x;u^{\varepsilon_{0}},v^{\varepsilon_{0}}}_{s})-\varepsilon_{0}\
|\ \mathcal {F}_{t}\Big)\geq 1-\varepsilon_{0}, \ \mathbb{P}-a.s.,
\end{eqnarray}
and
\begin{eqnarray}\label{e13}
|\mathbb{E}[J_{j}(t,x;u^{\varepsilon_{0}},v^{\varepsilon_{0}})]-
e_{j}|\leq\varepsilon_{0}.
\end{eqnarray}
We fix some partition: $t=t_{0}\leq t_{1}\leq\cdots \leq t_{m}=T$ of
$[t,T]$ and put $\tau=\sup\limits_{i}|t_{i}-t_{i+1}|$. Let us  apply
Lemma \ref{le1} to $u^{\varepsilon_{0}}$ and
$t+\delta=t_{1},\cdots,t_{m}$, successively. Then, for
$\varepsilon_{1}>0$ ($\varepsilon_{1}$ depends on $\varepsilon$ and
is specified later) we have the existence of  NAD strategies
$\alpha_{i}\in \mathcal {A}_{t,T}, i=1,\cdots,m$,  such that,  for
all
 $v\in \mathcal {V}_{t,T}$,
\begin{eqnarray}\label{eq8}
\alpha_{i}(v)&=&u^{\varepsilon_{0}}, \text{on}\ [t,t_{i}],\nonumber\\
^{2}Y^{t,x;\alpha_{i}(v),v}_{t_{i}}&\leq&
W_{2}(t_{i},X^{t,x;\alpha_{i}(v),v}_{t_{i}})+\varepsilon_{1},
\mathbb{P}-a.s.
\end{eqnarray}
For all $v\in \mathcal {V}_{t,T}$, we set
\begin{eqnarray*}\label{}
S^{v}&=&\inf\Big\{s\geq t \ |\ \lambda(\{r\in[t,s]: v_{r}\neq
v_{r}^{\varepsilon_{0}}\})>0\Big\},\\
t^{v}&=&\inf\Big\{t_{i}\geq S^{v} \ |\ i=1,\cdots, m\Big\}\wedge T.
\end{eqnarray*}
Here $\lambda$ denotes the Lebesgue measure on the real line
$\mathbb{R}$. We see that  $S^{v}$ and $t^{v}$ are stopping times
such that  $S^{v}\leq t^{v}\leq S^{v}+\tau$.

We put
\begin{equation*}\label{}
\alpha_{\varepsilon}(v)=\left\{
\begin{array}{ll}
 u^{\varepsilon_{0}}, & \ \text{on}\ [[ t, t^{v}]],\\
\alpha_{i}(v), & \ \text{on}\ (t_{i},T]\times \{t^{v}=t_{i}\}, 1\leq
i\leq m.
\end{array}
\right.
\end{equation*}
Then,  $\alpha_{\varepsilon}$ is an NAD strategy. It follows from
(\ref{eq8}) that
\begin{eqnarray}\label{eq9}
^{2}Y^{t,x;\alpha_{\varepsilon}(v),v}_{t^{v}}&=&
\sum\limits_{i=1}^{m} \
^{2}Y^{t,x;\alpha_{\varepsilon}(v),v}_{t_{i}}1_{\{t^{v}=t_{i}\}}\nonumber\\
&\leq&
\sum\limits_{i=1}^{m}W_{2}(t_{i},X^{t,x;\alpha_{\varepsilon}(v),v}_{t_{i}})1_{\{t^{v}=t_{i}\}}
+\varepsilon_{1}\nonumber\\
&=&
W_{2}(t^{v},X^{t,x;\alpha_{\varepsilon}(v),v}_{t^{v}})+\varepsilon_{1},
\ \mathbb{P}-a.s.
\end{eqnarray}
Let us  show that, for all $\varepsilon>0$ and $v\in \mathcal
{V}_{t,T}$,
\begin{eqnarray}\label{e1}
 J_{2}(t,x;\alpha_{\varepsilon}(v),v)
\leq J_{2}(t,x;u^{\varepsilon_{0}},v^{\varepsilon_{0}})
+\varepsilon, \quad
\alpha_{\varepsilon}(v^{\varepsilon_{0}})=u^{\varepsilon_{0}}.
\end{eqnarray}
Thanks to (\ref{eq9}), from Lemmas \ref{l18} and \ref{l8} we see
that there exists a positive constant $C$ such that
\begin{eqnarray}\label{eqn2}
 J_{2}(t,x,\alpha_{\varepsilon}(v),v)&=&
 ^{2}G^{t,x;\alpha_{\varepsilon}(v),v}_{t,t^{v}}[^{2}Y^{t,x,\alpha_{\varepsilon}(v),v}_{t^{v}}]\nonumber\\
&\leq &
^{2}G^{t,x;\alpha_{\varepsilon}(v),v}_{t,t^{v}}[W_{2}(t^{v},X^{t,x;\alpha_{\varepsilon}(v),v}_{t^{v}})
+\varepsilon_{1}]\nonumber\\
&\leq &
^{2}G^{t,x;\alpha_{\varepsilon}(v),v}_{t,t^{v}}[W_{2}(t^{v},X^{t,x;\alpha_{\varepsilon}(v),v}_{t^{v}})]
+C\varepsilon_{1}.
\end{eqnarray}
 Therefore, from Lemma \ref{l18}
\begin{eqnarray*}\label{}
&&|\
^{2}G^{t,x;\alpha_{\varepsilon}(v),v}_{t,t^{v}}[W_{2}(t^{v},X_{t^{v}}^{t,x;u^{\varepsilon_{0}},v^{\varepsilon_{0}}})]
-\ ^{2}G^{t,x;\alpha_{\varepsilon}(v),v}_{t,t^{v}}[W_{2}(t^{v},X^{t,x;\alpha_{\varepsilon}(v),v}_{t^{v}})]|\\
&\leq&C\mathbb{E}[|W_{2}(t^{v},X_{t^{v}}^{t,x;u^{\varepsilon_{0}},v^{\varepsilon_{0}}})
-W_{2}(t^{v},X^{t,x;\alpha_{\varepsilon}(v),v}_{t^{v}})|^{2}
\Big|\mathcal {F}_{t}]^{\frac{1}{2}}\\
 &\leq & C \mathbb{E}[|X_{t^{v}}^{t,x;u^{\varepsilon_{0}},v^{\varepsilon_{0}}}
-X^{t,x;\alpha_{\varepsilon}(v),v}_{t^{v}}|^{2}
\Big|\mathcal {F}_{t}]^{\frac{1}{2}}\\
 &\leq & C\tau^{\frac{1}{2}},\  \mathbb{P}-a.s.,
\end{eqnarray*}
for the last two inequalities we have used Proposition  \ref{pro}
and Lemma \ref{le2}. Then,  (\ref{eqn2}) yields
\begin{eqnarray*}\label{}
&& J_{2}(t,x,\alpha_{\varepsilon}(v),v)\\
&\leq &
^{2}G^{t,x;\alpha_{\varepsilon}(v),v}_{t,t^{v}}[W_{2}(t^{v},X_{t^{v}}^{t,x;u^{\varepsilon_{0}},v^{\varepsilon_{0}}})]
+C\varepsilon_{1}\\
&&+|\
^{2}G^{t,x;\alpha_{\varepsilon}(v),v}_{t,t^{v}}[W_{2}(t^{v},X_{t^{v}}^{t,x;u^{\varepsilon_{0}},v^{\varepsilon_{0}}})]
- \ ^{2}G^{t,x;\alpha_{\varepsilon}(v),v}_{t,t^{v}}[W_{2}(t^{v},X^{t,x;\alpha_{\varepsilon}(v),v}_{t^{v}})]|\\
&\leq &
^{2}G^{t,x;\alpha_{\varepsilon}(v),v}_{t,t^{v}}[W_{2}(t^{v},X_{t^{v}}^{t,x;u^{\varepsilon_{0}},v^{\varepsilon_{0}}})]
+C\varepsilon_{1}+C\tau^{\frac{1}{2}}.
\end{eqnarray*}
Putting
\begin{eqnarray}\label{e3}
\Omega_{s}=\Big\{\
^{2}Y^{t,x;u^{\varepsilon_{0}},v^{\varepsilon_{0}}}_{s}\geq
W_{2}(s,X^{t,x;u^{\varepsilon_{0}},v^{\varepsilon_{0}}}_{s})-\varepsilon_{0}\Big\},
s\in[t,T],
\end{eqnarray}
 we have
\begin{eqnarray}\label{eqn4}
&& J_{2}(t,x;\alpha_{\varepsilon}(v),v)\nonumber\\
 &\leq& \ ^{2}G^{t,x;\alpha_{\varepsilon}(v),v}_{t,t^{v}}
[\sum\limits_{i=1}^{m}W_{2}(t_{i},X_{t_{i}}^{t,x;u^{\varepsilon_{0}},v^{\varepsilon_{0}}})1_{\{t^{v}=t_{i}\}}]
+C\varepsilon_{1}+C\tau^{\frac{1}{2}}\nonumber\\
&\leq&^{2}G^{t,x;\alpha_{\varepsilon}(v),v}_{t,t^{v}}
[\sum\limits_{i=1}^{m}W_{2}(t_{i},X_{t_{i}}^{t,x;u^{\varepsilon_{0}},v^{\varepsilon_{0}}})
1_{\{t^{v}=t_{i}\}}1_{\Omega_{t_{i}}}]+C\varepsilon_{1}+C\tau^{\frac{1}{2}}
+I,
\end{eqnarray}
where $$I=|\ ^{2}G^{t,x;\alpha_{\varepsilon}(v),v}_{t,t^{v}}
[\sum\limits_{i=1}^{m}W_{2}(t_{i},X_{t_{i}}^{t,x;u^{\varepsilon_{0}},v^{\varepsilon_{0}}})1_{\{t^{v}=t_{i}\}}]
-\ ^{2}G^{t,x;\alpha_{\varepsilon}(v),v}_{t,t^{v}}
[\sum\limits_{i=1}^{m}W_{2}(t_{i},X_{t_{i}}^{t,x;u^{\varepsilon_{0}},v^{\varepsilon_{0}}})
1_{\{t^{v}=t_{i}\}}1_{\Omega_{t_{i}}}]|.$$
 Since $\Phi_{2},f_{2}$ and $h_{2}$ are bounded, from Lemma \ref{l1} we conclude that $W_{2}$ is bounded.
Therefore,
\begin{eqnarray}\label{eqn5}
I&\leq&\mathbb{E}[\sum\limits_{i=1}^{m}|W_{2}(t_{i},X_{t_{i}}^{t,x;u^{\varepsilon_{0}},v^{\varepsilon_{0}}})|^{2}
1_{\{t^{v}=t_{i}\}}1_{\Omega^{c}_{t_{i}}} \Big|\mathcal
{F}_{t}]^{\frac{1}{2}}\nonumber\\
&\leq& C\sum\limits_{i=1}^{m} \mathbb{P}(\Omega^{c}_{t_{i}}
|\mathcal{F}_{t})^{\frac{1}{2}}\leq Cm\varepsilon_{0}^{\frac{1}{2}},
\end{eqnarray}
where we have used (\ref{e12}) for the latter estimate.  From the
Lemmas \ref{l18},  \ref{l8} and (\ref{e3}) we have
\begin{eqnarray*}\label{}
&&^{2}G^{t,x;\alpha_{\varepsilon}(v),v}_{t,t^{v}}
[\sum\limits_{i=1}^{m}W_{2}(t_{i},X_{t_{i}}^{t,x;u^{\varepsilon_{0}},v^{\varepsilon_{0}}})
1_{\{t^{v}=t_{i}\}}1_{\Omega_{t_{i}}}]\\
&\leq& ^{2}G^{t,x;\alpha_{\varepsilon}(v),v}_{t,t^{v}}
[\sum\limits_{i=1}^{m}(^{2}Y^{t,x;u^{\varepsilon_{0}},v^{\varepsilon_{0}}}_{t_{i}}+\varepsilon_{0})
1_{\{t^{v}=t_{i}\}}1_{\Omega_{t_{i}}}]\\
&\leq& ^{2}G^{t,x;\alpha_{\varepsilon}(v),v}_{t,t^{v}}
[\sum\limits_{i=1}^{m}\
^{2}Y^{t,x;u^{\varepsilon_{0}},v^{\varepsilon_{0}}}_{t_{i}}
1_{\{t^{v}=t_{i}\}}1_{\Omega_{t_{i}}}+\varepsilon_{0}]\\
&\leq& ^{2}G^{t,x;\alpha_{\varepsilon}(v),v}_{t,t^{v}}
[\sum\limits_{i=1}^{m}\
^{2}Y^{t,x;u^{\varepsilon_{0}},v^{\varepsilon_{0}}}_{t_{i}}
1_{\{t^{v}=t_{i}\}}1_{\Omega_{t_{i}}}]+C\varepsilon_{0},
\end{eqnarray*}
and using the above arguments we also have
\begin{eqnarray*}\label{}
|\ ^{2}G^{t,x;\alpha_{\varepsilon}(v),v}_{t,t^{v}}
[\sum\limits_{i=1}^{m}\
^{2}Y^{t,x;u^{\varepsilon_{0}},v^{\varepsilon_{0}}}_{t_{i}}
1_{\{t^{v}=t_{i}\}}1_{\Omega_{t_{i}}}] -\
^{2}G^{t,x;\alpha_{\varepsilon}(v),v}_{t,t^{v}}
[\sum\limits_{i=1}^{m} \
^{2}Y^{t,x;\alpha_{\varepsilon}(v),v}_{t_{i}}1_{\{t^{v}=t_{i}\}}]|
\leq Cm\varepsilon_{0}^{\frac{1}{2}}.
\end{eqnarray*}
Consequently,
\begin{eqnarray*}\label{}
&&^{2}G^{t,x;\alpha_{\varepsilon}(v),v}_{t,t^{v}}
[\sum\limits_{i=1}^{m}W_{2}(t_{i},X_{t_{i}}^{t,x;u^{\varepsilon_{0}},v^{\varepsilon_{0}}})
1_{\{t^{v}=t_{i}\}}1_{\Omega_{t_{i}}}]\\
&\leq& ^{2}G^{t,x;\alpha_{\varepsilon}(v),v}_{t,t^{v}} [
^{2}Y^{t,x;u^{\varepsilon_{0}},v^{\varepsilon_{0}}}_{t^{v}}]+C\varepsilon_{0}+ Cm\varepsilon_{0}^{\frac{1}{2}}\\
&\leq&|\ ^{2}G^{t,x;\alpha_{\varepsilon}(v),v}_{t,t^{v}} [ \
^{2}Y^{t,x;u^{\varepsilon_{0}},v^{\varepsilon_{0}}}_{t^{v}}] -\
^{2}G^{t,x;u^{\varepsilon_{0}},v^{\varepsilon_{0}}}_{t,t^{v}} [ \
^{2}Y^{t,x;u^{\varepsilon_{0}},v^{\varepsilon_{0}}}_{t^{v}}]|\\&& +
\ ^{2}G^{t,x;u^{\varepsilon_{0}},v^{\varepsilon_{0}}}_{t,t^{v}} [ \
^{2}Y^{t,x;u^{\varepsilon_{0}},v^{\varepsilon_{0}}}_{t^{v}}]
+C\varepsilon_{0}+ Cm\varepsilon_{0}^{\frac{1}{2}}\\
&=&|\ ^{2}G^{t,x;\alpha_{\varepsilon}(v),v}_{t,t^{v}} [ \
^{2}Y^{t,x;u^{\varepsilon_{0}},v^{\varepsilon_{0}}}_{t^{v}}] -\
^{2}G^{t,x;u^{\varepsilon_{0}},v^{\varepsilon_{0}}}_{t,t^{v}} [
\ ^{2}Y^{t,x;u^{\varepsilon_{0}},v^{\varepsilon_{0}}}_{t^{v}}]|\\
&&+J_{2}(t,x;u^{\varepsilon_{0}},v^{\varepsilon_{0}})
+C\varepsilon_{0}+ Cm\varepsilon_{0}^{\frac{1}{2}}\\
&\leq&J_{2}(t,x;u^{\varepsilon_{0}},v^{\varepsilon_{0}})
+C\varepsilon_{0}+
Cm\varepsilon_{0}^{\frac{1}{2}}+C\tau^{\frac{1}{2}},
\end{eqnarray*}
where we have used the fact that
\begin{eqnarray*}\label{}
|\ ^{2}G^{t,x;\alpha_{\varepsilon}(v),v}_{t,t^{v}} [ \
^{2}Y^{t,x;u^{\varepsilon_{0}},v^{\varepsilon_{0}}}_{t^{v}}] -\
^{2}G^{t,x;u^{\varepsilon_{0}},v^{\varepsilon_{0}}}_{t,t^{v}} [ \
^{2}Y^{t,x;u^{\varepsilon_{0}},v^{\varepsilon_{0}}}_{t^{v}}]|\leq
C\tau^{\frac{1}{2}}.
\end{eqnarray*}
Indeed, let us consider the  following BSDE
\begin{eqnarray*}\label{}
    \left \{\begin{array}{rl}
      & y_s = \ ^{2}Y^{t,x;u^{\varepsilon_{0}},v^{\varepsilon_{0}}}_{t^{v}}+
\displaystyle\int_{s}^{t^{v}}f_{2}(r,X^{t,x;\alpha_{\varepsilon}(v),v}_r,
y_r,z_r,\alpha_{\varepsilon}(v_r), v_r)dr
                   +k_{t^{v}}-k_s -\displaystyle\int_{s}^{t^{v}}z_r dB_r,\\
      & y_{s}\geq
      h_{2}(s,X^{t,x;\alpha_{\varepsilon}(v),v}_s),  \qquad \qquad  s\in [t,t^{v}],\\
     &k_{t}=0, \qquad  \displaystyle \int_{t}^{t^{v}}(y_{r}-
      h_{2}(r,X^{t,x;\alpha_{\varepsilon}(v),v}_r))dk_{r}=0,
     \end{array}\right.
 \end{eqnarray*}
which compared with
\begin{eqnarray*}\label{}
    \left \{\begin{array}{rl}
      & \ ^{2}Y^{t,x;u^{\varepsilon_{0}},v^{\varepsilon_{0}}}_s
      =  \ ^{2}Y^{t,x;u^{\varepsilon_{0}},v^{\varepsilon_{0}}}_{t^{v}}+
\displaystyle\int_{s}^{t^{v}}
f_{2}(r,X^{t,x;u^{\varepsilon_{0}},v^{\varepsilon_{0}}}_r,\
^{2}Y^{t,x;u^{\varepsilon_{0}},v^{\varepsilon_{0}}}_r,\
^{2}Z^{t,x;u^{\varepsilon_{0}},v^{\varepsilon_{0}}}_r,u^{\varepsilon_{0}}_r,
v^{\varepsilon_{0}}_r)dr\\ & \hskip4cm +\
^{2}K^{t,x;u^{\varepsilon_{0}},v^{\varepsilon_{0}}}_{t^{v}}-\
^{2}K^{t,x;u^{\varepsilon_{0}},v^{\varepsilon_{0}}}_s
-\displaystyle\int_{s}^{t^{v}}\
^{2}Z^{t,x;u^{\varepsilon_{0}},v^{\varepsilon_{0}}}_rdB_r,\\
      & ^{2}Y^{t,x;u^{\varepsilon_{0}},v^{\varepsilon_{0}}}_{s}\geq
      h_{2}(s,X^{t,x;u^{\varepsilon_{0}},v^{\varepsilon_{0}}}_s),  \qquad \qquad  s\in [t,t^{v}],\\
     &^{2}K^{t,x;u^{\varepsilon_{0}},v^{\varepsilon_{0}}}_{t}=0,\quad
     \displaystyle\int_{t}^{t^{v}} (^{2}Y^{t,x;u^{\varepsilon_{0}},v^{\varepsilon_{0}}}_{r}-
      h_{2}(r,X^{t,x;u^{\varepsilon_{0}},v^{\varepsilon_{0}}}_r))
      d\ ^{2}K^{t,x;u^{\varepsilon_{0}},v^{\varepsilon_{0}}}_{r}=0,
     \end{array}\right.
 \end{eqnarray*}
Note that  $\alpha_{\varepsilon}(v)=u^{\varepsilon_{0}}, $ on $[[ t,
t^{v}]]$, $v=v^{\varepsilon_{0}}, $ on $[[t, S^{v}]]$, and from
Lemma \ref{l18} we obtain
\begin{eqnarray*}
&&|\ ^{2}G^{t,x;\alpha_{\varepsilon}(v),v}_{t,t^{v}} [
^{2}Y^{t,x;u^{\varepsilon_{0}},v^{\varepsilon_{0}}}_{t^{v}}] - \
^{2}G^{t,x;u^{\varepsilon_{0}},v^{\varepsilon_{0}}}_{t,t^{v}} [
^{2}Y^{t,x;u^{\varepsilon_{0}},v^{\varepsilon_{0}}}_{t^{v}}]|^{2}\\
&\leq & C \mathbb{E}[\int_t^{t^{v}}
|f_{2}(r,X^{t,x;\alpha_{\varepsilon}(v),v}_r,y_r,z_r,\alpha_{\varepsilon}(v)_r,
v_r)dr-f_{2}(r,X^{t,x;u^{\varepsilon_{0}},v^{\varepsilon_{0}}}_r,
y_r, z_r,u^{\varepsilon_{0}}_r, v^{\varepsilon_{0}}_r) |^2|\mathcal{F}_t]\\
&&+C \mathbb{E}[\sup\limits_{r\in[t,t^{v}]}
|h_{2}(r,X^{t,x;\alpha_{\varepsilon}(v),v}_r)-h_{2}(r,X^{t,x;u^{\varepsilon_{0}},v^{\varepsilon_{0}}}_r)
|^2|\mathcal{F}_t]^{\frac{1}{2}}\\
&=&C\mathbb{E}[\int_{S^{v}}^{t^{v}}
|f_{2}(r,X^{t,x;\alpha_{\varepsilon}(v),v}_r,y_r,z_r,\alpha_{\varepsilon}(v)_r,
v_r)dr-f_{2}(r,X^{t,x;u^{\varepsilon_{0}},v^{\varepsilon_{0}}}_r,
y_r, z_r,u^{\varepsilon_{0}}_r, v^{\varepsilon_{0}}_r) |^2|\mathcal{F}_t]\\
&&+C \mathbb{E}[\sup\limits_{r\in[S^{v},t^{v}]}
|X^{t,x;\alpha_{\varepsilon}(v),v}_r-X^{t,x;u^{\varepsilon_{0}},v^{\varepsilon_{0}}}_r|^2|\mathcal{F}_t]^{\frac{1}{2}}\\
 &\leq & C \mathbb{E}[\int_{S^{v}}^{t^{v}} 1_{\{v_r\neq
v^{\varepsilon_{0}}_r\}} |\mathcal {F}_t]+ C\tau^{\frac{1}{2}}\leq C
\mathbb{E}[t^{v}-S^{v}|\mathcal {F}_t]+ C\tau^{\frac{1}{2}}\leq
C\tau^{\frac{1}{2}},
\end{eqnarray*}
where we have used the boundedness of $f_{2}, b$ and $\sigma$.
Consequently,
\begin{eqnarray*}\label{}
&&^{2}G^{t,x;\alpha_{\varepsilon}(v),v}_{t,t^{v}}
[\sum\limits_{i=1}^{m}W_{2}(t_{i},X_{t_{i}}^{t,x;u^{\varepsilon_{0}},v^{\varepsilon_{0}}})
1_{\{t^{v}=t_{i}\}}1_{\Omega_{t_{i}}}]\\
&&\leq C\tau^{\frac{1}{2}}
+J_{2}(t,x;u^{\varepsilon_{0}},v^{\varepsilon_{0}})
+C\varepsilon_{0}+Cm\varepsilon_{0}^{\frac{1}{2}},
\end{eqnarray*}
and thus, (\ref{eqn4}) and  (\ref{eqn5}) yield
\begin{eqnarray*}\label{}
 J_{2}(t,x;\alpha_{\varepsilon}(v),v)
\leq J_{2}(t,x;u^{\varepsilon_{0}},v^{\varepsilon_{0}})
+C\varepsilon_{0}+Cm\varepsilon_{0}^{\frac{1}{2}}
+C\varepsilon_{1}+C\tau^{\frac{1}{4}}.
\end{eqnarray*}
We can choose $\tau>0, \varepsilon_{0}>0,$ and $ \varepsilon_{1}>0$
such that $C\varepsilon_{0}+Cm\varepsilon_{0}^{\frac{1}{2}}
+C\varepsilon_{1}+C\tau^{\frac{1}{4}}\leq \varepsilon$ and
$\varepsilon_{0}<\varepsilon.$ Thus,
\begin{eqnarray*}\label{}
 J_{2}(t,x;\alpha_{\varepsilon}(v),v)
\leq J_{2}(t,x;u^{\varepsilon_{0}},v^{\varepsilon_{0}})
+\varepsilon, v\in \mathcal {V}_{t,T}.
\end{eqnarray*}
Using  a symmetric argument we can construct $\beta_{\varepsilon}\in
\mathcal {B}_{t,T}$ such that, for all $u\in \mathcal {U}_{t,T}$,
\begin{eqnarray}\label{e2}
 J_{1}(t,x;u,\beta_{\varepsilon}(u))
\leq J_{1}(t,x;u^{\varepsilon_{0}},v^{\varepsilon_{0}})
+\varepsilon,\quad
\beta_{\varepsilon}(u^{\varepsilon_{0}})=v^{\varepsilon_{0}}.
\end{eqnarray}
Finally, from   (\ref{e1}), (\ref{e2}), (\ref{e13}) and Lemma
\ref{le3} it follows that
$(\alpha_{\varepsilon},\beta_{\varepsilon})$ satisfies Definition
\ref{d2}. Hence, $(e_{1},e_{2})$ is a Nash equilibrium payoff.
\end{proof}

\end{document}